\documentclass[a4paper,12pt]{amsart}

\usepackage[margin=2.3cm]{geometry}

\usepackage[utf8]{inputenc}
\usepackage[T1]{fontenc}
\usepackage[english]{babel}

\usepackage{amsmath}
\usepackage{amssymb}
\usepackage{amsthm}

\usepackage{hyperref} 

\theoremstyle{plain}
\newtheorem{thm}{Theorem}[section]
\newtheorem{lem}[thm]{Lemma}
\newtheorem{prop}[thm]{Proposition}
\newtheorem{coro}[thm]{Corollary}

\theoremstyle{definition}
\newtheorem{deff}[thm]{Definition}
\newtheorem{exe}[thm]{Example}

\theoremstyle{remark}

\DeclareMathOperator{\tr}{tr} 
\DeclareMathOperator{\sing}{sing} 
\DeclareMathOperator{\vol}{vol} 

\newcommand{\abs}[1]{\left\lvert#1\right\rvert} 
\newcommand{\norm}[1]{\left\lVert#1\right\rVert} 

\newcommand{\Z}{\mathbb{Z}} 
\newcommand{\R}{\mathbb{R}} 
\newcommand{\C}{\mathbb{C}} 
\newcommand{\Hip}{\mathbb{H}} 
\newcommand{\Proj}{\mathbb{P}} 
\newcommand{\s}{\mathbb{S}} 
\newcommand{\D}{\mathbb{D}} 
\newcommand{\T}{\mathbb{T}} 
\newcommand{\Cinf}{\mathcal{C}^\infty} 

\newcommand{\cal}[1]{\mathcal{#1}} 
\renewcommand{\Re}{\operatorname{Re}} 
\renewcommand{\Im}{\operatorname{Im}} 
\renewcommand{\d}{\mathop{}\!\mathrm{d}} 

\begin{document}

\title{Levi-flat hypersurfaces and their complement in complex surfaces}
\author{Carolina \textsc{Canales González}}
\address{Laboratoire de Mathématiques d'Orsay, Univ. Paris-Sud, CNRS, Université Paris-Saclay, 91405 Orsay, France}
\curraddr{Anillo PIA-CONICYT ACT1415, Universidad de la Frontera, Departamento de Ma\-te\-má\-ti\-ca y Estadística, Temuco, Chile}
\email{carolina.canales@ufrontera.cl}

\begin{abstract}
In this work we study analytic Levi-flat hypersurfaces in complex algebraic surfaces. First, we show that if this foliation admits chaotic dynamics (i.e. if it does not admit a transverse invariant measure), then the connected components of the complement of the hypersurface are modifications of Stein domains. This allows us to extend the CR foliation to a singular algebraic foliation on the ambient complex surface. We apply this result to prove, by contradiction, that analytic Levi-flat hypersurfaces admitting a transverse affine structure in a complex algebraic surface have a transverse invariant measure. This leads us to conjecture that Levi-flat hypersurfaces in complex algebraic surfaces that are diffeomorphic to a hyperbolic torus bundle over the circle are fibrations by algebraic curves.
\end{abstract}

\maketitle

\section{Introduction}

In this work, we study real analytic hypersurfaces in complex algebraic surfaces that satisfy a certain partial differential equation that we describe here below. Given a real hypersurface $M$ in a complex surface $X$, we define the Cauchy-Riemann distribution on $M$, called in abbreviated form CR distribution, that in a point $p$ of $ M $ is the unique complex line contained in $T_pM$, i.e. the distribution $TM\cap iTM$. The hypersurface $M$ is called Levi-flat if the CR distribution is integrable. This means that through any point of $M$ passes a non-singular holomorphic curve of $X$ that is completely contained in $M$. These curves correspond then to the leaves of a foliation on $M$, noted $\cal F$ hereafter, called the Cauchy-Riemann foliation or CR foliation. The condition that ensures a real hypersurface to be Levi-flat can be synthesized by the vanishing of the Levi form. The purpose of this work is to understand the interaction between the dynamics of the CR foliation, the topology of the hypersurface, and the geometry of its complement in the ambient surface.

\vspace{1em}

Before stating our results, let us give some examples of Levi-flat hypersurfaces, which we organize following the dynamic complexity of their CR foliation. This gives us the opportunity to introduce some terminology borrowed from the theory of dynamical systems.

\vspace{1em}

\paragraph*{\textbf{Periodic Levi-flat Hypersurfaces.}} These hypersurfaces are those fibered by algebraic curves. They appear in all birational equivalence classes of algebraic surfaces. In fact, any complex algebraic surface admits a pencil that, after a finite number of blow-ups, becomes a singular fibration. A family of fibers of such a fibration parametrized by the circle describes a periodic Levi-flat hypersurface. The topology of these Levi-flat hypersurfaces can be very rich: all Thurston's geometries except the spherical geometry are realized by Levi-flat hypersurfaces of this type. Certain compact quotients of each of the following models appear thus as a periodic Levi-flat hypersurface
\[
\s^2\times\R,\ \R^3,\ \mathrm{Nil},\ \mathrm{Sol},\ \Hip^2\times\R, \ \widetilde{\mathrm{SL}(2,\R)},\ \Hip^3.
\]
We refer to \cite{Deroin-Dupont} for more details on these constructions.

\vspace{1em}

\paragraph*{\textbf{Quasi-periodic Levi-flat hypersurfaces.}} Emblematic examples of these are linear hypersurfaces in a complex torus. In this case, the CR foliation is a linear foliation of a real torus of dimension three. In general, we will say that a Levi-flat hypersurface is \emph{quasi-periodic} if its CR foliation is defined (up to a double cover) by a closed differential form, or equivalently as the pre-image by a smooth map of a codimension 1 linear foliation on a torus. Examples of such hypersurfaces appear in line bundles of degree zero on a curve, and therefore in the neighbourhood of linearizable curves in complex surfaces. Arnol'd has provided such examples in the blow-up of $\C\Proj^2$ in nine generic points with respect to the Lebesgue measure \cite[§ 27]{Arnold}.

\vspace{1em}

\paragraph*{\textbf{Chaotic Levi-flat hypersurfaces.}} These are the hypersurfaces for which the CR foliation admits no transverse invariant measure. A transverse invariant measure is a family of Borel measures on the transversals to the foliation that are invariant under any holonomy map. To get an idea of these objects, a property satisfied by chaotic Levi-flat hypersurfaces is the existence of a leaf with hyperbolic holonomy, see \cite{Deroin-Kleptsyn}: some leaves wind around others with exponential rate. Important examples of such hypersurfaces are constructed in flat $\C\Proj^1$-bundles with real monodromy over a curve: just consider the corresponding $\R\Proj^1$-bundle. \footnote{In particular, when the monodromy of the $\C\Proj^1$-bundle is that given by the uniformization of the base curve $C$, the CR foliation of the previously considered hypersurface is analytically conjugated to the weak stable foliation of the geodesic flow on $C$ equipped with its conformal metric with curvature $-1$. Such a flow is the emblematic example of a chaotic flow, reinforcing the terminology used here.}

\vspace{1em}

These three classes do not describe all Levi-flat hypersurfaces in complex algebraic surfaces. For example, there are Levi-flat hypersurfaces containing algebraic curves, but that are not periodic. This is the case of the Levi-flat hypersurfaces constructed by Nemirovski\u{\i}, see \cite{Nemirovskii}. Furthermore, although we know no example of this kind, one could imagine that there are analytic Levi-flat hypersurfaces that admit a non atomic transverse invariant measure but that is singular with respect to the Lebesgue measure. Such hypersurfaces would be similar to quasi-periodic Levi-flat hypersurfaces, yet still different.

\vspace{1em}

An important part of this work concerns the study of the geometrical properties of the connected components of the complement of a Levi-flat hypersurface. These exterior components are by definition pseudoconvex and, as we shall see, it is interesting to understand how this local convexity property globalizes. Here are several ideas of global convexity that we present in ``increasing order'':

\begin{enumerate}
\item \label{itm:weakpsconv} Weakly pseudoconvex: there is a plurisubharmonic exhaustion function.
\item \label{itm:holosconv} Holomorphically convex: the holomorphic convex envelope of a compact is compact.
\item \label{itm:stpsconv} Strongly pseudoconvex: there is an exhaustion function that is strictly plurisubharmonic outside a compact.
\item \label{itm:stein} Stein: it is holomorphically convex and holomorphic functions separate points.
\end{enumerate}

It is well known that in the case of a domain with $\cal C^2$ boundary in a compact complex manifold, each of these properties implies the previous one. The implication \ref{itm:stein} $\Rightarrow$ \ref{itm:stpsconv} is valid for any complex manifold. This follows from the fact that a complex manifold is Stein if and only if it admits a strictly plurisubharmonic exhaustion function, according to a theorem of Grauert, see \cite{Grauert}. The implication \ref{itm:stpsconv} $\Rightarrow$ \ref{itm:holosconv} is a theorem due to Grauert \cite{Grauert} and Narasimhan \cite{Narasimhan}: it shows in particular that, by Remmert's reduction, a strongly pseudoconvex domain is a modification of a Stein variety. We see then that \ref{itm:stpsconv} and \ref{itm:stein} are very close: we pass from one to the other by a blow-up procedure.

\vspace{1em}

It turns out that global convexity properties of the exterior components of an algebraic Levi-flat hypersurface are intimately related to the dynamics of the CR foliation. For instance, a famous example of Grauert in \cite{Grauert-modifications} shows that the exterior of a linear hypersurface in a complex torus is always pseudoconvex, and it is holomorphically convex when the CR foliation is periodic. This phenomenon generalizes to quasi-periodic Levi-flat hypersurfaces. However, exterior components of quasi-periodic Levi-flat hypersurfaces are never strongly pseudoconvex. Indeed, in these components there are holomorphic curves or Levi-flat hypersurfaces arbitrarily close to their boundary, which doesn't allow the strictly subharmonicity of an exhaustion function near the boundary.

\vspace{1em}

The main result of this work states that \textit{external components of (analytic) chaotic Levi-flat hypersurfaces are strongly pseudoconvex}, i.e. modifications of a Stein space by the result of Grauert and Narasimhan mentioned above. This result was known in particular cases, including that of the hypothetical Levi-flat hypersurfaces of the complex projective plane by Takeuchi's theorem, see \cite{Takeuchi-projectif}, or flat bundles in $\C\Proj^1$ over a curve of genus $\geq 2$ whose monodromy is real, faithful and discrete (i.e. associated to the uniformization of a curve with the same genus of the base), according to a theorem of Diederich and Ohsawa \cite{Diederich-Ohsawa}. It is interesting to remark that our condition, although quite general, is not optimal: in the very interesting article \cite{Nemirovskii}, Nemirovski\u{\i} defines Levi-flat hypersurfaces in some elliptical surfaces that cut the surface into Stein domains. These Levi-flat hypersurfaces contain invariant elliptic curves and therefore are not chaotic. However they admit similar turbulence properties to those satisfied by chaotic Levi-flat hypersurfaces. It would be interesting to understand what is the optimal condition for this problem, but we do not pursue such a study in this paper.

\vspace{1em}

To prove our result, we analyze the geometry of neighbourhoods of chaotic Levi-flat hypersurfaces, building positive curvature metrics on the normal bundle to the foliation. Our construction is based on the work of Deroin and Kleptsyn in \cite{Deroin-Kleptsyn}, where the heat equation along the leaves of the foliation is considered: it is shown that under the chaotic hypothesis, the leaves converge exponentially quickly towards each other along brownian foliated trajectories. It is this phenomenon that allows us to build the mentioned metric on the normal bundle of the CR foliation. A theorem of Brunella is used then to deduce the strongly convexity of the exterior components of the hypersurface, see \cite{Brunella-ample}.

\vspace{1em}

Our study of global convexity properties of the exterior components of a Levi-flat hypersurface has several consequences. The first one is the following rigidity property: \textit{every analytic chaotic Levi-flat hypersurface is tangent to a singular complex algebraic foliation defined on the ambient surface.} This fact follows from extension techniques of analytic objects in modified Stein spaces, which are now classic, see \cite{Siu-Trautman, LinsNeto, Ivashkovich-bochner, Merker-Porten}, but that we detail in this text, particularly the delicate passage through critical levels of the plurisubharmonic exhaustion function. Thus, chaotic Levi-flat hypersurfaces appear as regular invariant sets of algebraic differential equations. So this is a new manifestation of the GAGA principle (``Géométrie Analytique Géométrie Algébrique'') in the context of Levi-flat hypersurfaces, but whose validity requires dynamic assumptions, unlike the conventional case. Indeed, without the chaotic hypothesis, Sad noticed that there are counterexamples: the quasi-periodic Levi-flat hypersurfaces constructed by Arnol'd in the blow-up of $\C\Proj^2$ in nine generic points are not tangent to a complex algebraic foliation.

\vspace{1em}

It seems reasonable to think that the algebraic differential equations constrained to preserve a real analytic set are sufficiently rare to be classifiable. As a comparison, in the theory of iteration, we know all the rational applications of a complex variable whose Julia sets are analytical: there are only Tschebychev polynomials and Blaschke products. This interesting problem, certainly a difficult one, will not be considered in this generality here. However, we can use the above-mentioned rigidity properties to study particular classes of Levi-flat hypersurfaces. This technique allows us to understand the structure of transversely affine analytic Levi-flat hypersurfaces, i.e. whose holonomy pseudogroup is given in an analytical coordinate by affine transformations of the form $x\in\R\mapsto ax+b\in\R$, with $a\in\R^*$ and $b\in\R$. We show, by combining our rigidity result with a theorem of Ghys, that \textit{a transversely affine Levi-flat hypersurface in a complex algebraic surface is quasi-periodic, or contains an algebraic curve}. This result echoes the recent work of Pereira and Cousin on the classification of singular transversely affine algebraic foliations of codimension 1, see \cite{Cousin-Pereira}.

\vspace{1em}

Our study of transversely affine analytic Levi-flat hypersurfaces has interesting consequences for the geometry of analytic Levi-flat hypersurfaces that are diffeomorphic to a hyperbolic torus bundle. These manifolds are bundles with toric fibers $\T^2=\R^2/\Z^2$ whose monodromy is isotopic to a linear automorphism $A\in\mathrm{GL}(2,\Z)$ with eigenvalues of norm different from 1. We conjecture that \textit{analytic Levi-flat hypersurfaces of complex algebraic surfaces that are diffeomorphic to a hyperbolic torus bundle are periodic}. We are not able to prove this conjecture, \footnote{Note here that Proposition 2 in \cite{Ghys-Sergiescu} would demonstrate this conjecture, but Étienne Ghys warned us about the existence of a counterexample to this statement, which he kindly let us reproduce in this text.} but we do have a result that goes in this direction: we show that \textit{analytic Levi-flat hypersurfaces in complex algebraic surfaces that are diffeomorphic to a hyperbolic torus bundle contain an elliptic curve}. This gives in particular an alternative proof of a recent theorem of Deroin and Dupont \cite{Deroin-Dupont} in the analytic case: \textit{analytic Levi-flat hypersurfaces in surfaces of general type have a non solvable fundamental group}. The proof of these results is immediate if one remembers the beautiful classification of codimension 1 analytical foliations on hyperbolic torus bundles, due to Ghys and Sergiescu \cite{Ghys-Sergiescu}. These foliations are of two types: either they are analytically conjugated to the suspension of the stable or unstable foliation associated to the matrix $A$ (in which case they are chaotic and transversely affine, which is excluded for the CR foliation of a Levi-flat in an algebraic surface), or they admit a compact leaf that is a torus and isotopic to a fiber of the bundle.

\vspace{1em}

\paragraph*{\textbf{Acknowledgements}} I would like to thank my advisors, Betrand Deroin and Christophe Dupont, and their institutions, for their help and the discussions about this work. I would also like to thank the referee for carefully reading this manuscript and the suggestions made. 

\section{Preliminaries}

In this section, we give some well-known definitions and results in order to fix notations. These notations will be kept throughout all the text.

\subsection{Foliations}

\begin{deff}\label{def:feuilletage3reel}
Let $M$ be a compact analytic 3-manifold. A foliation $\cal F$ by Riemann surfaces on $M$ is an atlas  $\cal A=\{(U_j,\varphi_j)\}_{j\in J}$ for $M$ that is maximal with respect to the following properties:
\begin{enumerate}
\item For all $j\in J$, $\varphi_j:U_j\to A_j\times B_j$ is an analytic diffeomorphism, where $A_j$ is an open disc in $\C$ and $B_j=]0,1[$.
\item If $(U_j,\varphi_j)$ and $(U_k,\varphi_k)$ belong to $\cal A$ with $U_j\cap U_k\neq\emptyset$, then
\[\varphi_{jk}:=\varphi_j\circ\varphi_k^{-1}:\varphi_k(U_j\cap U_k)\to\varphi_j(U_j\cap U_k)\]
is of the form
\[\varphi_{jk}(z_k,t_k)=(f_{jk}(z_k,t_k),g_{jk}(t_k)),\]
where $f_{jk}$ and $g_{jk}$, are analytic functions and $f_{jk}$ depends holomorphically of $z_k$.
\end{enumerate}
\end{deff}

We will note $N_{\cal F}$ and $T_{\cal F}$ the \emph{tangent} and \emph{normal bundle} of the foliation respectively, which are line bundles given by the transition functions $\left\{\frac{\d g_{jk}}{\d t_k}\right\}$ and $\left\{\frac{\d f_{jk}}{\d z_k}\right\}$ respectively. We define also the \emph{conormal bundle} to the foliation $N^*_{\cal F}$ by taking $\left\{(\frac{\d g_{jk}}{\d t_k})^{-1}\right\}$ as transition functions.

A \emph{metric} $m$ on these bundles will be given locally in the chart $(U_j,(z_j,t_j,\xi_j))$ by a collection of functions 
\[m_j(z_j,t_j,\xi_j)=e^{-\sigma_j(z_j,t_j)}\abs{\xi_j}^2\] where $\sigma_j$ is a $\Cinf$ function and the \emph{curvature} of $m$ will be given locally by the (1,1)-form
\[\Theta_{m_j}=\frac{i}{2\pi}\partial\bar\partial_{\cal F}\sigma_j:=\frac{i}{2\pi}\frac{\partial^2\sigma_j}{\partial z\partial\bar z}.\]

We can also define a foliation on a compact 3-manifold by local submersions, vector fields or differential forms. This follows from Frobenius' Theorem.

\begin{thm}[Frobenius' Theorem]
Let $M$ be a compact analytic 3-manifold. Let $E=\{E_p\}_{p\in M}$ be an analytic distribution of planes and $\{\omega_p\}_{p\in M}$ be a family of differential 1-forms such that $\ker(\omega_p)=E_p$. The following conditions are equivalent:
\begin{enumerate}
\item For all $p\in M$, there exists a submanifold $N$ of $M$, such that $\iota_*(T_pN)=E_p$, where $\iota:N\to M$ is the natural inclusion.
\item For all vector fields $X_1,X_2$ in $E$ we have that $[X_1,X_2]$ belongs to $E$, where $[\cdot,\cdot]$ is the Lie bracket of vector fields.
\item For all $p\in M$ the differential form $\d\omega_p$ vanishes on $\ker(\omega_p)$, i.e. $\d\omega\wedge\omega=0$.
\item For all $p\in M$ there exists a neighbourhood $U$ of $p$ and a differential 1-form $\eta$ on $U$ such that $\d\omega=\eta\wedge\omega$.
\item For all $p\in M$ there exists a neighbourhood $U$ of $p$ and functions $f,g:U\to\R$ such that $\omega=f\d g$.
\end{enumerate}
\end{thm}

If one of the conditions is verified, we will say that the distribution of planes $E$ is Frobenius integrable.

Let $X$ be a complex surface. We define a \emph{holomorphic foliation} $\cal G$ on $X$ in a similar way as Definiton \ref{def:feuilletage3reel}: we change the transversal real coordinates $t_j\in]0,1[$ for complex coordinates $w_j\in\D$ and functions $f_{jk}, g_{jk}$ become holomorphic.

We define also in a similar way the \emph{normal bundle} $N_{\cal G}$, the \emph{conormal bundle} $N_{\cal G}^*$ and a metric $m$ on these bundles. The \emph{curvature} of $m$ will be given locally by the (1,1)-form
\[\Theta_{m_j}=\frac{i}{2\pi}\partial\bar\partial_X\sigma_j,\]
where $\partial_X:=\partial_z+\partial_w$ and $\bar\partial_X:=\partial_{\bar z}+\partial_{\bar w}$.

Propositions \ref{prop:defparsubmersions}, \ref{prop:deffeuilletageformediff} and \ref{prop:Frobeniusholo} are holomorphic versions of Frobenius' Theorem.
 
\begin{prop}\label{prop:defparsubmersions}
Let $\cal G$ be a non-singular holomorphic foliation. There exists a cover $\{U_j\}_{j\in J}$ of $X$ by open sets and a family of submersions $g_j:U_j\to\C$, such that if $U_j\cap U_k\neq\emptyset$ then $g_j=H_{jk}\circ g_k$, where $H_{jk}:g_k(U_j\cap U_k)\subset\C\to\C$ is a holomorphic function.
\end{prop}

\begin{deff}
A singular holomorphic foliation on $X$ is given by a non-singular holomorphic foliation on $X$ defined away from a finite set of points. We note $\sing(\cal G)$ this finite set and we call it the singular set of the foliation.
\end{deff}

The following proposition characterizes singular holomorphic foliations using holomorphic 1-forms. We remark that on a complex surface every holomorphic 1-form $\omega$ is integrable because $\d\omega\wedge\omega$ is a 3-form.

\begin{prop}\label{prop:deffeuilletageformediff}
Let $\cal G$ be a singular holomorphic foliation on $X$. There exists an open cover $\{U_j\}_{j\in J}$ of $X$ and a collection of holomorphic 1-forms $\omega_j$ on $U_j$ with isolated zeros such that if $U_j\cap U_k\neq\emptyset$ then
\[\omega_j=h_{jk}\,\omega_k\, ,\text{ where } h_{jk}\in\cal O^*(U_j\cap U_k).\]
Moreover, the set $\sing(\cal G)$ is equal to $\cup_{j\in J}\{\omega_j=0\}$.
\end{prop}

The following proposition is immediate.

\begin{prop}\label{prop:Frobeniusholo}
Let $\cal G$ be a non-singular holomorphic foliation on $X$. Let $g_j:U_j\to\C$ be a familly of holomorphic submersions such that $g_j=H_{jk}\circ g_k$ as in Proposition \ref{prop:defparsubmersions}. Then we can take $\omega_j=\d g_j$ and $h_{jk}=H'_{jk}$ in Proposition \ref{prop:deffeuilletageformediff}.
\end{prop}

To finish this section we define the notion of meromorphic 1-form defining a singular foliation on a complex surface $X$.

\begin{deff}\label{def:feuilletageformemero}
Let $\cal G$ be a singular holomorphic foliation on $X$ defined by a collection of holomorphic 1-forms $\{\omega_j\}_{j\in J}$ as in Proposition \ref{prop:deffeuilletageformediff}. We say that a meromorphic 1-form $\omega$ on $X$ defines $\cal G$ if, for all $j\in J$ and for all $p\in U_j$ that is not a zero nor a pole of $\omega$, we have $\ker\omega(p)=\ker\omega_j(p)$.
\end{deff}

Finally, we explain how the existence of meromorphic sections of the conormal bundle $N_{\cal G}^*$ allows us to construct meromorphic 1-forms on $X$ that define the foliation in the sense of Definition \ref{def:feuilletageformemero}. This will be useful in Section \ref{sec:feuilletagesaffinesdeg}.

\begin{prop}\label{prop:sectmerom}
Let $X$ be a complex surface and let $\cal G$ be a singular holomorphic foliation on $X$. Let $\{U_j\}_{j\in J}$ be an open cover of $X\setminus\sing(\cal G)$ and $g_j:U_j\to\C$ be submersions that define $\cal G$ as in Proposition \ref{prop:defparsubmersions}. Let $f=(f_j)$ be a meromorphic section of $N^*_{\cal G}$. Then the collection of meromorphic 1-forms $\omega_j:=f_j\d g_j$ glues into a meromorphic 1-form on $X$ that defines the foliation $\cal G$ in the sense of Definition \ref{def:feuilletageformemero}.
\end{prop}

\subsection{Levi-flat hypersurfaces}\label{sec:LeviPlats}

Levi-flat hypersurfaces are real hypersurfaces of complex surfaces that are foliated by Riemann surfaces. The complex structure of the leaves is of course compatible with the complex structure of the ambient surface. Their name is justified by the vanishing of the Levi-form.

\begin{deff}\label{def:leviplat}
Let $X$ be a complex surface and $M\subset X$ be a compact analytic 3-manifold. Let $T_pM\subset T_pX$ be the tangent space to $M$ in the point $p$. We say that $M$ is a Levi-flat hypersurface if the distribution of planes $E_p:=T_pM\cap iT_pM$ on $M$ is Frobenius integrable. We note $\cal F$ the corresponding foliation by Riemann surfaces on $M$ and we call it the Cauchy-Riemann foliation, or CR foliation, of $M$.
\end{deff}

In the following, we give some examples that illustrate each of the notions of Levi-flat hypersurfaces mentioned in the introduction: periodic, quasi-periodic and chaotic.

\begin{exe}
Our first example is a periodic Levi-flat in an algebraic surface that is diffeomorphic to a hyperbolic torus bundle. Here is the idea of its construction, which is detailed in \cite{Deroin-Dupont}.

We consider a pencil of cubics in $\C\Proj^2$, for example the pencil of cubics that in affine coordinates are defined by
\[
C_t:=\overline{\{y^2=x(x-1)(x-t)\}},
\]
where $t\in\C\Proj^1$. These cubics are smooth if $t\neq\{0,1,\infty\}$. They intersect all in nine distinct points $p_1,\ldots,p_9$ of the projective complex plane. Blowing-up $\C\Proj^2$ in these nine points, we obtain a rational surface $X$ provided with a singular fibration whose fibers are the strict transforms $\hat{C_t}$ of the curves $C_t$ in $X$.

Let us now move the variable $t$ along a loop in $\C\Proj^1\setminus\{0,1,\infty\}$. The union of curves $\hat{C_t}$ in $X$ is an immersed Levi-flat hypersurface that is diffeomorphic to a torus bundle by construction. If the curve described by $t$ is complicated enough, then the monodromy of this bundle is hyperbolic. We obtain a periodic Levi-flat hypersurface diffeomorphic to a hyperbolic torus bundle in a rational surface, but a priori this hypersurface is only immersed if the curve described by $t$ is not simple.

To obtain an embedded Levi-flat hypersurface of this type, we have to take the preimage of this construction by a cover over $\C\Proj^1$, whose only ramification points are $0,1,\infty$, and in which the curve described by $t$ transforms into a simple loop.
\end{exe}

\begin{exe}
Our second example is a perturbation of the first one due to Arnol'd, see \cite{Arnold}.

In order to describe it, we will need the following notion: a curve $C$ in a complex surface $X$ is linearizable if the inclusion of $C$ in its normal bundle extends to a biholomorphism from a neighbourhood of $C$ to a neighbourhood of its image in the normal bundle. If a smooth curve with trivial self-intersection in a complex surface is linearizable, then there exist quasi-periodic Levi-flat hypersurfaces arbitrarily near $C$. Indeed, it suffices to consider the zero section of the normal bundle: we can take then the levels of the unique flat metric on it.

If we choose nine points $p_1,\ldots,p_9$ generically in $\C\Proj^2$, then there exists a unique cubic $C$ passing through these nine points. This cubic is of self-intersection 9, so the strict transform $\hat{C}$ of $C$ in the blow-up $X$ of $\C\Proj^2$ along the points $p_1,\ldots,p_9$ has trivial self-intersection. In \cite{Arnold} Arnol'd shows that this curve is linearizable. Then there are quasi-periodic Levi-flat hypersurfaces in its neighbourhood. A remark of Sad shows that these hypersurfaces are never tangent to an algebraic foliation of $X$.
\end{exe}

\begin{exe}
Let $\Sigma$ be a compact Riemann surface of genus $\geq2$. Let $\rho:\pi_1(\Sigma)\to\mathrm{PSL}(2,\R)$ be a representation of the fundamental group of $\Sigma$ in the group of real Moebius transformations. We define a complex surface $X$ as the quotient of $\D\times\C\Proj^1$ by the relation $(z,t)\sim(\gamma\cdot z,\rho(\gamma)t)$ for all $\gamma\in\pi_1(\Sigma)$. We note $X:=\Sigma\ltimes_\rho\C\Proj^1$. Let $p:\D\times\C\Proj^1\to X$ be the projection to the quotient space $X$. Then $M:=p(\D\times\s^1)=\Sigma\ltimes_\rho\R\Proj^1$ is a Levi-flat hypersurface in $X$. As we will see, this Levi-flat hypersurface is chaotic for generic representations $\rho$.
\end{exe}

\section{Dynamics of Levi-flat hypersurfaces}\label{sec:dynamique}

In this section we will prove that if $\cal F$ is a foliation by Riemann surfaces on a compact 3-manifold without transverse invariant measure, then the normal bundle $N_{\cal F}$ of the foliation has a metric with positive curvature. It is the crucial part of this work.

\subsection{Dynamics on foliated 3-manifolds}

Let $M$ be a compact 3-manifold and $\cal F$ a codimension 1 foliation on $M$. Recall that a \emph{transverse invariant measure} is a family of finite measures $\{\mu_{\cal T}\}_{\cal T}$ on each transversal $\cal T$ of the foliation $\cal F$ such that every holonomy map $h_\gamma:\mathrm{dom}(h_\gamma)\subset\cal T_0\to\cal T_1$ verifies $\mu(h_\gamma (B))=\mu(B)$ on the borelian subsets of $\mathrm{dom}(h_\gamma)$. An invariant transverse measure is \emph{ergodic} if it cannot be written as a convex non trivial combination of two different transverse invariant measures.

It is not common for a foliation to have a transverse invariant measure. We have the following result in the case of Fuchsian representations.

\begin{exe}\label{exe:fibre_plat_monodromie_reelle}
Let $\Sigma$ be a Riemann surface and $\rho:\pi_1(\Sigma)\to\mathrm{PSL}(2,\R)$ a representation of the fundamental group of $\Sigma$. Define $M:=\Sigma\ltimes_\rho\R\Proj^1$. Then there exists a transverse invariant measure on $M$ if and only if the representation $\rho$ is elementary, i.e. the image of $\pi_1(\Sigma)$ by $\rho$ is conjugated to
\begin{itemize}
\item a subgroup of $\mathrm{SO}(2)$, or
\item a subgroup of the affine group, or
\item a subgroup of the group generated by $z\mapsto\frac 1z$ and $z\mapsto az$, where $a\in\R^*$.
\end{itemize}
In the first case the invariant transverse measure is the Lebesgue measure, in the other cases it is a sum of two Dirac masses. See for example the book of Beardon \cite[§5.1]{Beardon}.
\end{exe}

One can't always find invariants measures for a foliation, but we can always find \emph{harmonic measures}, see \cite[Chapter 2.1]{FoliationsII}. These were introduced by Garnett \cite{Garnett}. If $\mathbf g$ is a metric on the tangent bundle $T_{\cal F}$ of the foliation with laplacian $\Delta_{\mathbf g}$, a probability measure $\mu$ on $M$ is $\mathbf g$-harmonic if for all $f\in\cal C_{\cal F}^\infty(M)$ we have $\int_M\Delta_{\mathbf g}f(x)\,\d\mu(x)=0.$

These measures generalize in some sense transverse invariant measures: given a harmonic measure, we can construct a family of transverse measures that are, on average, invariant by holonomies with respect to the heat distribution on the leaves. We can construct these measures by hand in some cases, see \cite{FoliationsII} for the case of hyperbolic torus bundles and \cite{Garnett} for the case of unit tangent bundle of hyperbolic Riemann surfaces.

We can associate to an ergodic harmonic measure a number, called Lyapunov exponent, which measures the exponential rate of separation of leaves. In \cite{Deroin-Kleptsyn} Deroin and Kleptsyn prove that for an ergodic $\mathbf g$-harmonic measure $\mu$ on $M$ the \emph{Lyapunov exponent} of $\mu$ is
\[\lambda(\mu)=-\int_M\varphi\,\d\mu,\]
where $\varphi:M\to\R$ is the function verifying $\Theta_{m_{\cal F}}=\varphi\cdot\vol_{\mathbf g}$ and $m_{\cal F}$ is a metric on the normal bundle $N_{\cal F}$.

In \cite{Deroin-Kleptsyn} also, they prove that the Lyapunov exponent of an invariant measure is zero and that it is negative when the foliation doesn't have an invariant measure:

\begin{thm}\cite{Deroin-Kleptsyn}\label{thm:corothmb}
Let $M$ be a compact 3-manifold foliated by Riemann surfaces and $\mathbf g$ a metric on $T_{\cal F}$. If $\cal F$ doesn't have a transverse invariant measure then $\cal F$ has a finite number of minimal sets $\cal M_1,\ldots,\cal M_r$, each supports a unique $\mathbf g$-harmonic measure $\nu_i$ and each Lyapunov exponent $\lambda(\nu_i)$ is $<0$. Every harmonic measure on $M$ is a convex combination of $\nu_1,\ldots,\nu_r$.
\end{thm}

\subsection{Dynamics and positive curvature}\label{sec:deroinmodifie}

In this section we show that if a compact 3-manifold foliated by Riemann surfaces has no transverse invariant measure then the normal bundle $N_{\cal F}$ has a metric whose curvature is positive. We begin with the following lemma, whose proof is inspired by \cite{Deroin-Kleptsyn}.

\begin{lem}\label{lem:enleverintegrale}
Let $M$ be a compact 3-manifold and $\cal F$ a foliation by Riemann surfaces. Let $\mathbf g$ be a metric on the tangent bundle $T_{\cal F}$. Let $\varphi:M\to\R$ be a continuous function such that $\int_M \varphi \, d\nu>0$ for all ${\mathbf g}$-harmonic measure $\nu$ on $M$. Then there exists $g\in\Cinf_{\cal F}(M)$ such that $\varphi-\Delta_{\mathbf g} g>0$.
\end{lem}

\begin{proof}
We consider the space $C^0(M)$ of continuous functions $\phi:M\to\R$. It is a Banach space for the  topology of uniform convergence. We note $E$ the closed subspace
\[E :=\{\phi\in C^0(M)\mid\exists (g_n)_n \in\Cinf_{\cal F}(M)\text{ such that } \norm{\phi-\Delta_{\mathbf g} g_n}_\infty\to0\}\]
and $C$ the cone
\[C :=\{\psi\in C^0(M)\mid\psi>0 \text{ on } M\}.\]
The set $C$ is a convex cone because for every $\alpha>0$ and every function $\psi\in C$, the function $\alpha\psi$ is in $C$.
We note $F=C^0(M)/E$ and $\pi:C^0(M)\to F$ the canonical projection. Since $E$ is closed, $\norm{\pi(\phi)}:=\inf_{e\in E}\norm{\phi-e}_\infty$ defines a norm on the space $F$, hence $F$ is a Banach space. Moreover the projection $\pi$ is linear and continuous.

Let $\varphi\in C^0(M)$ such that $\int_M \varphi\d\nu>0$ for every harmonic measure $\nu$.
To show that there exists $g\in\Cinf_{\cal F}(M)$ such that $\varphi-\Delta_{\mathbf g} g>0$, it is enough to prove that $\pi(\varphi)\in\pi(C)$.

Indeed, if $\pi(\varphi)\in\pi(C)$ then there exists a function $c\in C$ and a function $e\in E$ such that $\varphi=c+e$. Since $M$ is compact and $c$ is a continuous function on $M$ there exists $\epsilon>0$ such that $c\geq\epsilon$. Therefore we have $\varphi-e\geq\epsilon$ on $M$. According to the definition of $E$ there exist $g\in\Cinf_{\cal F}(M)$ such that $\abs{e-g}_\infty<\epsilon/2$. Consequently we have $\varphi-\Delta_{\mathbf g} g\geq\epsilon/2$ as desired.

So we have to prove that $\pi(\varphi)\in\pi(C)$. By contradiction we suppose that $\pi(\varphi)\not\in\pi(C)$. The set $\pi(C)$ is convex and open in $F$. We use the classical Hahn-Banach theorem for $\pi(C)$ and the closed subset $\{\pi(\varphi)\}$ of $F$. According to this theorem there exists a linear continuous functional $L:F\to\R$ and $\alpha\in\R$ such that $L(\pi(\varphi))\leq\alpha$ and $L\geq\alpha$ on $\pi(C)$. We verify that this gives us
\begin{equation}\label{eq:alphaescero}
L(\pi(\varphi))\leq0 \text{ and } L\geq0 \text{ on } \pi(C).
\end{equation}
If we take a sequence of constant functions $\epsilon_n\in C$ such that $\epsilon_n\to0$, then since $L$ and $\pi$ are linear and continuous we have $\alpha\leq L(\pi(\epsilon_n))\to L(\pi(0))=0$. Hence $\alpha\leq0$ and then $L(\pi(\varphi))\leq0$. To verify $L\circ\pi\geq0$ on $C$, we suppose by contradiction that there exists $g\in C$ such that 
$L(\pi(g))=\beta<0$. Since $C$ is a cone, and $L$ and $\pi$ are linear, for all $\lambda\in\R$ we have
\[L(\pi(\lambda g))=\lambda L(\pi(g))=\lambda\beta\to -\infty, \text{ if } \lambda\to +\infty\]
and this is a contradiction with $L(\pi(C))\geq\alpha$. This shows \eqref{eq:alphaescero}.

We define a linear functional $\tilde L$ on the space $C^0(M)$ by $\tilde L(\phi):=L(\pi(\phi))$. It is positive since, by definition of $\tilde L$ on $C$, we have $\tilde L(\psi)=L(\pi(\psi))>0$ for every function $\psi\in C$. Since $\tilde L$ is positive, the Riesz Representation Theorem tells us that there exists a positive measure $\nu$ representing $\tilde L$, i.e. $\tilde L(\phi)=\int_M \phi\d\nu$ for all $\phi\in C^0(M)$.
This measure is $\mathbf g$-harmonic because for every $\Cinf$-function $g$ we have $\tilde L(\Delta_{\mathbf g} g)=L(\pi(\Delta_{\mathbf g} g))=0$, since $\Delta_{\mathbf g} g \in E$.

Thus, we have constructed a harmonic measure $\nu$ such that
\[\int_M \varphi \, \d\nu=L(\pi(\varphi))=\tilde L(\varphi)\leq0.\]
But by hypothesis $\int_M \varphi \, d\nu>0$ has to be positive. This contradiction shows that $\pi(\varphi)\in\pi(C)$.
\end{proof}

\begin{thm}\label{thm:metriquecourburepositive}
Let $M$ be a compact 3-manifold and let $\cal F$ be a foliation by Riemann surfaces on $M$. If $\cal F$ doesn't have a transverse invariant measure then $N_{\cal F}$ has a metric with positive curvature.
\end{thm}

\begin{proof}
Let $\tilde m$ be a metric on the normal bundle $N_{\cal F}$ and $\mathbf g$ a metric on the tangent bundle $T_{\cal F}$. Let $\varphi:M\to\R$ be the continuous function such that $\Theta_{m_{\cal F}}=\varphi\cdot\vol_{\mathbf g}$.
According to Theorem \ref{thm:corothmb}, since the foliation $\cal F$ doesn't have a transverse invariant measure, every $\mathbf g$-harmonic measure $\nu$ supported on $M$ is a convex combination $\nu=\sum_i t_i\nu_i$, with $\nu_i$ a ergodic harmonic measure of Lyapunov exponent $\lambda(\nu_i)<0$. In particular, we have 
\[-\int_M \varphi \, \d\nu = - \sum_i t_i \int_M \varphi \, \d\nu_i = \sum_i t_i \lambda(\nu_i) < 0.\]
Then Lemma \ref{lem:enleverintegrale} gives us a $\Cinf$-function $g$ such that
\[\varphi-\Delta_{\mathbf g} g>0 \text{ on } M.\]
We define
\[m_{\cal F}= \tilde m \cdot e^g.\]
The curvature of this metric on $N_{\cal F}$ verifies
\[\Theta_{m_{\cal F}}=\Theta_{\tilde m}-\frac{i}{2\pi}\partial\bar\partial_{\cal F} g=\left(\varphi - \Delta_{\mathbf g} g\right) \vol_{\mathbf g}.\]
Hence the curvature of $m_{\cal F}$ is positive because $\varphi-\Delta_{\mathbf g} g$ is positive on $M$. 
\end{proof}

We end this section with some particular examples where we can construct explicitly a metric on the normal bundle of the CR foliation with positive curvature.

\begin{exe}\label{exe:mesure_lisse}
Let $\cal F$ be a foliation by Riemann surfaces of a compact $3$-manifold, and $\mathbf g$ a metric on $T\cal F$. Let $\mu$ be a harmonic measure. We suppose that $\mu$ is absolutely continuous with respect to the Lebesgue measure, and that its density is a continuous function that is everywhere strictly positive. In this case, we can split the harmonic measure into the product of the volume form along the leaves associated to the metric $\mathbf g$ and a transverse volume element. This transverse volume allows us to build a metric on the normal bundle of the foliation. The measure $\mu$ being harmonic, we get that the norm of a non zero plane section of the normal bundle, with respect to the Bott connection, is a positive function and it is harmonic along the leaves. Hence the curvature of this metric is non negative, since the laplacian of the logarithm of a positive harmonic function is non positive.
\end{exe}

The existence of a harmonic measure that is absolutely continuous with respect to the Lebesgue measure is not very common. Indeed, it could happen that $\cal F$ has minimal exceptional sets, and in this case $\mu$ is supported on the union of these minimals \cite{Deroin-Kleptsyn}. These minimals are conjectured to be of Lebesgue measure zero. Moreover, even in the case where the foliation is minimal, the harmonic measures will often be singular with respect to the Lebesgue measure. We refer to the papers \cite{Deroin-Kleptsyn-Navas-circle1,Deroin-Kleptsyn-Navas-circle2} for these delicate questions. However, there exist some particular cases where the harmonic measures are smooth and have a positive density. This is the case for homogeneous foliations for example. 

\begin{exe}
Let $G$ be a Lie group of dimension 3. We suppose that $G$ contains a lattice $\Gamma$ and a copy of the affine group
\[A=\{x\mapsto ux+v \mid u\in\R^*_+, v\in\R\}.\]
Let $M:=\Gamma\backslash G$ and $\cal F$ the foliation of $M$ defined by the locally free action of $A$ on $M$ by the left product given by
\[a\cdot\Gamma g := \Gamma g a^{-1} \text{ for every } a\in A, \ g\in G.\]

We consider on $A$ the metric $\mathbf g=\frac{\d u^2+\d v^2}{u^2}$ with constant curvature $-1$. We verify that it is invariant by left multiplication. We remark that $\mathbf g$ is not invariant by right multiplication, nor is its volume. If $D_a:a'\in A\mapsto a'a^{-1}\in A$ is the right product by $a^{-1}$, we have
\begin{equation}\label{eq:equivariancedroite}
D_a^*\vol_{\mathbf g}=u\vol_{\mathbf g}.
\end{equation}
The orbits of the action of $A$ on $M$ are quotients of $A$ by a discrete subgroup of $A$ acting on the left. Consequently $\mathbf g$ endows the leaves of $\cal F$ with a riemannian metric with constant curvature $-1$.

Let $\tilde{\mu}$ be a Haar measure on $G$. This measure is bi-invariant by $\Gamma$, hence it descends to a volume measure $\mu$ on $M$ that is invariant by the action of the affine group $A$. This form decomposes into the product $\mu =\d\vol_{\mathbf g}\wedge\omega$, where $\omega$ is a form that gives the foliation $\cal F$. By relation \eqref{eq:equivariancedroite}, and the invariance of $\mu$ by right multiplication, we obtain
\[D_a^*\omega=\frac 1u\omega, \text{ for every } a\in A.\]
In particular, if $s$ is a plane section of the normal bundle of $\cal F$ along the leaves, we have $\omega(s)=ku$, with $k$ a constant, for any parametrisation of the leaf seen as an orbit of the affine group. This shows that the curvature of the metric $m_\cal F:=\abs{\omega}$ on the normal bundle of $\cal F$ is given by $\Theta_{m_\cal F}=\frac{1}{4\pi}\vol_{\mathbf g}$. In particular, it is positive.

We remark that in the case where $G$ is the group $\mathrm{Sol}:=\R_{>0}\ltimes\R^2$, where $\R_{>0}$ acts on $\R^2$ by $u\cdot(v,w):=(uv,u^{-1}w)$, the manifolds $M$ that we obtain in this way are, up to a finite cover, hyperbolic torus bundles. We send the reader to \cite{Garnett} for a more geometric view-point in the case where $G=\mathrm{PSL}(2,\R)$.
\end{exe}

\section{Geometry of the complement}\label{sec:geocp}

In this section we prove the following
\begin{thm}\label{thm:complementfortementpconvLP}
Let $X$ be a compact complex surface, $M$ a compact real analytic Levi-flat hypersurface in $X$. If $M$ doesn't have a transverse invariant measure, then the connected components of $X\setminus M$ are modifications of Stein spaces.
\end{thm}

This result was known in some particular cases, in particular in the case of hypothetical Levi-flat hypersurfaces of the complex projective plane by a theorem of Takeuchi \cite{Takeuchi-projectif}, or in the case of flat $\C\Proj^1$-bundles over a curve of genus $\geq 2$ with real, faithful and discrete monodromy (i.e. associated to the uniformisation of a curve with the same genus as the base), by a theorem of Diederich and Ohsawa \cite{Diederich-Ohsawa}. It is interesting to remark that our condition, although quite general, is not optimal: in the interesting work \cite{Nemirovskii}, Nemirovski\u{\i} defines Levi-flat hypersurfaces in some elliptic surfaces, that separate the surface in Stein domains. These Levi-flat hypersurfaces contain some elliptic invariant curves, and hence they are not chaotic. However they have turbulent properties similar to those satisfied by chaotic Levi-flat hypersurfaces. It would be interesting to understand what is the optimal condition for this problem, but we don't push this question further in this work.

To prove our result, the idea is to construct an exhaustion function on the complement of the Levi-flat hypersurface that is strictly pseudoconvex outside of a compact set. Doing this we obtain that $X\setminus M$ is strongly pseudoconvex. Using a result of Grauert \cite{Grauert-modifications} we get that $X\setminus M$ is a modification of a Stein domain. We follow the following strategy:
\begin{enumerate}
\item Construct a metric with positive curvature on the normal bundle of the CR foliation.
\item Extend the CR foliation locally.
\item Extend the metric with positive curvature locally.
\item Construct a strictly pseudoconvex exhaustion function on the complement.
\end{enumerate}

The first step is Theorem \ref{thm:metriquecourburepositive}, which we proved in Section \ref{sec:dynamique}. The second step is a well known result that we can find in \cite{LinsNeto} or \cite{Cartan-LeviPlat}:

\begin{prop}\label{prop:extentionvoisinage}
Let $X$ be a complex surface, $M$ a real analytic Levi-flat hypersurface in $X$ and $\cal F$ the Cauchy-Riemann foliation of $M$. Then, there exists a neighbourhood $U$ of $M$ in $X$ and a non-singular holomorphic foliation $\cal G$ on $U$ such that $\cal G|_M=\cal F$.
\end{prop}

The last step is a result of Brunella:

\begin{thm}\label{thm:complementopseudoconvexo}\cite{Brunella-ample}
Let $X$ be a compact complex surface and $M$ a real analytic Levi-flat hypersurface in $X$. We suppose that $M$ is invariant by a foliation $\cal G$ defined on some neighbourhood $U$ of $M$ in $X$. Moreover, we suppose that the normal bundle $N_{\cal F}$ of the foliation has a metric with positive curvature. Then there exists a neighbourhood $U'\subset U$ of $M$ in $X$ and a strictly plurisubharmonic function $h:U'\to(-\infty,+\infty]$ such that
\begin{enumerate}
\item $h(p)\to+\infty$ when $p\to M$.
\item $h$ is an exhaustion function on $U'\setminus M$.
\end{enumerate}
\end{thm}

We remark that Brunella proves this result in $\cal C^{2,\alpha}$ regularity. The idea is to construct several local strictly plurisubharmonic exhaustion functions on the complement of $M$ and the difficult part is to glue them together.

The only step that we are left then in order to prove Theorem \ref{thm:complementfortementpconvLP} is the third. We will extend the metric on $N_{\cal F}$ with positive curvature along the leaves to a metric on the normal bundle of the extended foliation with positive curvature in every direction. The continuity of the curvature will be needed.

But before doing this, let us give some examples of applications of Theorem \ref{thm:complementfortementpconvLP}.

\newpage

\subsection{Examples of modified Stein complements}\label{sec:exemples}

\mbox{}\vspace{1em}

\paragraph{\textbf{$\C\Proj^1$-bundles with real monodromy}}
In this paragraph we consider the flat bundles over a curve $\Sigma$ associated to the real representation $\rho:\pi_1(\Sigma)\to\mathrm{PSL}(2,\R)$ that we introduced in Example \ref{exe:fibre_plat_monodromie_reelle}.

In this case, the hypersurface $M:=\Sigma\ltimes_\rho\R\Proj^1$ of $X:=\Sigma\ltimes_\rho\C\Proj^1$ is a Levi-flat hypersurface and cuts $X$ into two domains $D^\pm:=\Sigma\ltimes_\rho\Hip^\pm$, where $\Hip^\pm:=\{z\in\C\mid\pm\Im z>0\}$ are the upper and lower half-planes.\footnote{Remark that we could adapt our arguments to the case of representations with values in $\mathrm{PGL}(2,\R)$, which would allow us to produce examples of non orientable Levi-flat hypersurfaces that don't cut $X$ into two connected components $D^\pm$ but define a unique component $D$ instead. We will not be doing this in order to simplify the exposition.} We have seen that $M$ admits a transverse invariant measure only in the case where $\rho$ is elementary, see Example \ref{exe:fibre_plat_monodromie_reelle}. We obtain then

\begin{coro}\label{cor:exemples}
If $\rho$ is non elementary, then the domains $D^\pm$ are modifications of Stein domains.
\end{coro}

This result was shown by Diederich and Ohsawa in \cite{Diederich-Ohsawa} in the case where $\rho$ is a faithful and discrete representation with values in $\mathrm{PSL}(2,\R)$, i.e. corresponding to the uniformisation of a curve of the same genus as $\Sigma$. In an earlier article \cite{Diederich-Ohsawa-weak} they have shown, without conditions on $\rho$, that the domains $D^\pm$ are always weakly pseudoconvex. In the two cases, the argument lies on the existence of harmonic equivariant maps. In the case where the monodromy is faithful and discrete, the argument to show strongly pseudoconvexity lies on the Schoen-Yau theorem, which says that a harmonic equivariant map between two hyperbolic compact surfaces is a diffeomorphism \cite{Schoen-Yau}.

It is interesting to remark that the components $D^\pm$ are not always Stein surfaces, even though they are minimal (i.e. they do not contain a rational curve). Indeed, it can happen that these domains contain sections of the natural fibration $X\to\Sigma$.

To see this, take an integer $0\leq k<2g-2$, where $g$ is the genus of $\Sigma$, and $k$ points $p_1,\ldots,p_k$ of $\Sigma$ that we assume distinct for simplicity. A theorem of Troyanov assures then that there exists a unique hermitian metric $\mathbf{g}$ on $\Sigma\setminus\{p_1,\ldots,p_k\}$ of curvature $-1$ with conic singularities of angle $4\pi$ on the $p_i$'s. See \cite{Troyanov} for more details.

The fact that these conic angles are multiples of $2\pi$ shows that on the universal cover $\widetilde{\Sigma}$ of $\Sigma$, the metric $\mathbf{g}$ is the preimage of the Poincaré metric $\frac{\abs{\d z}^2}{(\Im z)^2}$ on the half-plane $\Hip^+$ by a holomophic map $D:\widetilde{\Sigma}\to\Hip^+$, i.e.
\[\pi^*\mathbf{g} = D^*\frac{\abs{\d z}^2}{(\Im z)^2}\]
where $\pi:\widetilde{\Sigma}\to\Sigma$ is the cover map. Since the metric $\pi^*\mathbf{g}$ is invariant by the fundamental group of $\Sigma$, the map $D$ is equivariant by some representation
\[\rho:\pi_1(\Sigma)\to\mathrm{Isom}^+\left(\Hip^+,\frac{\abs{\d z}}{(\Im z)^2}\right)\simeq\mathrm{PSL}(2,\R),\]
called the holonomy of the conic metric.

The $\rho$-equivariant map $D:\widetilde{\Sigma}\to\Hip^+$ defines a holomorphic section $\Sigma\to D^+=\Sigma\ltimes_\rho\Hip^+$. The representation $\rho$ is of Euler class $k-2g+2$ and by consequence it is non elementary. Hence, the domain $D^+$ is an example of a modification of a Stein domain that is minimal, but admitting an exceptional set containing a curve of genus $g\geq2$.

\vspace{1em}

\paragraph{\textbf{Torus bundles}}
This example, given by Nemirovski\u{\i} \cite{Nemirovskii}, is a Levi-flat hypersurface with Stein complement in a complex surface obtained by taking the quotient of a line bundle over an elliptic curve. What is interesting is that this Levi-flat has algebraic curves, showing that our Theorem \ref{thm:complementfortementpconvLP} is not optimal.

Let $\Sigma$ be a compact Riemann surface and $L\to\Sigma$ a holomorphic line bundle. We consider a meromorphic section $s:\Sigma\to L$ with only simple zeros and poles. We note $Z$ the zeroes and $P$ the poles and we suppose that there exists at least a zero or a pole. Let $\Sigma^*=\Sigma\setminus(Z\cup P)$. We note $L^*$ the fiber bundle $L$ without the zero section. On each fiber $L^*_x$ of $L^*$ over $x\in\Sigma^*$, we consider the real line passing by $s(x)$ that we note $l_x$. The set $\R s=\{l_x\}_{x\in\Sigma^*}$ of all lines is an analytic Levi-flat whose leaves are biholomorphic to $\Sigma^*$. The set $E:=\overline{\R s}$ is an analytic Levi-flat in the bundle $L^*$.

We consider the equivalence relation on $L^*$ given by $p\sim 2p$. We note $X:=L^*/\sim$ the quotient space that is then a torus bundle over $\Sigma$. The quotient of $\overline{\R s}$ by this relation is a Levi-flat hypersurface that cuts $X$ into two Stein domains. Indeed, every connected component of the complement is a trivial fibration in annuli over $\Sigma^*$. The annulus and the surface $\Sigma^*$ are both open Riemann surfaces and hence they are Stein by a theorem of Behnke and Stein. Finally, the product of two Stein spaces is Stein.

\subsection{Local extension of the metric with positive curvature }

To extend a metric with positive curvature along the leaves on $N_{\cal F}$ to a metric with positive curvature in every direction on $N_{\cal G}$, where $\cal F$ is the CR foliation and $\cal G$ is the local extension of $\cal F$, we will use the following lemmas. The first one is a classical result that we can find in \cite{LinsNeto}:

\begin{lem}[Local description of Levi-flats]\label{lem:formelocaleLeviPlat}
Let $X$ be a complex surface and $M$ a real analytic Levi-flat hypersurface in $X$. For all $p\in M$ there exist a neighbourhood $U_p$ of $p$ homeomorphic to a ball, and a holomorphic function $H=u+iv$ defined in $U_p$ such that $\d H(p)\neq0$ and $M\cap U_p=\{v=0\}$.
\end{lem}

The second one gives a condition for the Levi form of a function to be positive. We recall that the Levi form of a function $f$ defined on an open set $U\subset\C^2$ is the hermitian form over $\C^2$ defined by
\[L_f(p)(\zeta_1,\zeta_2)=
\frac{\partial^2f(p)}{\partial z\partial\bar z}\zeta_1\bar \zeta_1
+\frac{\partial^2f(p)}{\partial z\partial\bar w}\zeta_1\bar \zeta_2
+\frac{\partial^2f(p)}{\partial w\partial\bar z}\zeta_2\bar \zeta_1
+\frac{\partial^2f(p)}{\partial w\partial\bar w}\zeta_2\bar \zeta_2.\]
We say that $L_f$ is positive on $U$ if $L_f(p)(\zeta_1,\zeta_2)>0$ for all $p\in U$ and every $(\zeta_1,\zeta_2)\in\C^2\setminus\{(0,0)\}$. Moreover we have the relation
\[i\partial\bar\partial f\wedge i\alpha\wedge\bar\alpha=L_f(\alpha_1.\alpha_2)\vol_X\]
for every 1-form $\alpha=-\alpha_2\d z+\alpha_1\d w$. So the Levi form of $f$ is positive if and only if the 
(1,1)-form $i\partial\bar\partial f$ is positive.

\begin{lem}\label{lem:formeLeviestpositive}
If $\frac{\partial^2f(p)}{\partial z\partial\bar z} > 0$ and if $\abs{\frac{\partial^2f(p)}{\partial z\partial\bar w}}^2 < \frac{\partial^2f(p)}{\partial z\partial\bar z}\frac{\partial^2f(p)}{\partial w\partial\bar w}$ for all $p \in U$, then the (1,1)-form $i\partial\bar\partial_X f$ is positive on $U$.
\end{lem}

\begin{proof}
We remark that with these hypothesis $\frac{\partial^2f(p)}{\partial z\partial\bar z}>0$ if and only if $\frac{\partial^2f(p)}{\partial w\partial\bar w}>0$.
To prove the lemma it suffices to show that $L_f(p)>0$ on $\C^2\setminus\{(0,0)\}$. If $\alpha_1=0$ and $\alpha_2\neq0$ then $L_f(p)(\alpha_1,\alpha_2)=\abs{\alpha_2}^2\frac{\partial^2f}{\partial w\partial\bar w}$ and we are done. Now, if $\alpha_1\neq0$ then $L_f(p)(\alpha_1,\alpha_2)=\abs{\alpha_1}^2L_f(p)(1,\frac{\alpha_2}{\alpha_1})$. It is sufficient then to prove
\[L_f(p)(1,\alpha) = \frac{\partial^2f}{\partial z\partial\bar z}+2\Re\left(\frac{\partial^2f}{\partial w\partial\bar z}\alpha\right)+\frac{\partial^2f}{\partial w\partial\bar w}\abs{\alpha}^2 > 0\]
for all $\alpha\in\C$. We remark that
\[L_f(p)(1,\alpha)\geq\frac{\partial^2f}{\partial z\partial\bar z}-2\abs{\frac{\partial^2f}{\partial w\partial\bar z}}\abs{\alpha}+\frac{\partial^2f}{\partial w\partial\bar w}\abs{\alpha}^2.\]
This quadratic polynomial is positive because its first coefficient is and its discriminant is negative.
\end{proof}

Now we can extend our metric:

\begin{prop}\label{prop:extensionmetrique}
Let $X$ be a complex surface, $M$ a real analytic Levi-flat hypersurface in $X$ and $\cal F$ its Cauchy-Riemann foliation. We suppose that the normal bundle $N_{\cal F}$ has a metric with positive curvature in the direction of the leaves. Then there exists a neighbourhood $U'$ of $M$ in $X$ and a non-singular holomorphic foliation $\cal G$ on $U'$ such that $\cal G|_M=\cal F$ and the normal bundle $N_{\cal G}$ has a metric with positive curvature in every direction.
\end{prop}

\begin{proof}
We note $m_{\cal F}$ a metric on the normal bundle $N_{\cal F}$ on $M$ with positive curvature in the direction of the leaves.

By Proposition \ref{prop:extentionvoisinage} there exists a neighbourhood $U$ of $M$ in $X$ and a non-singular holomorphic foliation $\cal G$ defined on $U$ such that $\cal G|_M=\cal F$.

We take a cover of $U$ by charts $(U_j,(z_j,w_j))$ such that $M\cap U_j=\{\Im(w_j)=0\}$ as in Lemma \ref{lem:formelocaleLeviPlat}. In these charts we note $m_j:= m_{\cal F}|_{U_j}$, $m_j(z_j,t_j,\xi_j)=e^{-\sigma_j(z_j,t_j)}\abs{\xi_j}^2$.

We extend $m_j$ to $U_j$ by $\tilde m_j(z_j,w_j,\tilde\xi_j):=m_j(z_j,\Re(w_j))\abs{\tilde\xi_j}^2$. We consider a partition of unity $\{\phi_j\}_{j\in J}$ subordinate to the cover $\{U_j\}_{j\in J}$ and we define $\tilde m=\sum_j\phi_j\tilde m_j$. By continuity, this extends the metric $m_{\cal F}$ to a metric $\tilde m$ on the bundle $N_{\cal G}$ that has positive curvature in the direction of the leaves. To obtain a positive curvature in every direction we set
\[m_{\cal G}:=\tilde m\exp(-Cd_M^2)\]
where $d_M$ is the distance to $M$ (with respect to a fixed metric on $X$) and $C$ is a constant. The curvature of $m_{\cal G}$ is then equal to
\[\Theta_{m_{\cal G}}=\Theta_{\tilde m}+C{\frac{i}{2\pi}}\partial\bar\partial_X d_M^2,\]
that we will write in the following way
\[\frac{i}{2\pi}\partial\bar\partial_X\left[-\log(m_{\cal G})\right] = \frac{i}{2\pi}\partial\bar\partial_X\left[-\log(\tilde m\exp(-C d_M ^2))\right].\]
We will show that this $(1,1)$-form is positive when $C$ is big enough.
By Lemma \ref{lem:formeLeviestpositive} it is sufficient to verify the inequality 
\begin{multline*}
\abs{\frac{\partial^2}{\partial w\partial\bar z}\left[-\log(\tilde m)+Cd_M(\cdot)^2\right]}^2\\
< \left(\frac{\partial^2}{\partial z\partial\bar z}\left[-\log(\tilde m)
  +Cd_M(\cdot)^2\right]\right)\left(\frac{\partial^2}{\partial w\partial\bar w}\left[-\log(\tilde m)
  +Cd_M(\cdot)^2\right]\right)
\end{multline*}
on a neighbourhood of $M$. By continuity of these derivatives, it is sufficient to verify this inequality on $M$. Let $p\in U$ and $U_j$ be an open set of the cover of $U$ containing $p$. On $U_j$ we can write the distance $d_M$ in the following way
\[d_M^2(z_j(p),w_j(p))=\Im(w_j(p))^2 h(z_j(p),w_j(p))\]
where $h$ is smooth and positive. The laplacian in the direction of the leaves is then
\[\frac{\partial^2}{\partial z\partial\bar z}\left[-\log(\tilde m)+Cd_M(\cdot)^2\right]
= \frac{\partial^2}{\partial z\partial\bar z}[-\log(\tilde m)]
  +C\Im(w_j)^2\frac{\partial^2}{\partial z\partial\bar z}h(z_j,w_j).\]
In the point $p\in M$ the first term on the left hand side is a positive $(1,1)$-form, it is in fact the curvature of $\tilde m$ in the direction of the leaves. The second term is zero in $p\in M$ because $\Im(w_j)=0$. The laplacian in the direction of the leaves is then positive. Now, the transverse laplacian is equal to
\begin{multline*}
\frac{\partial^2}{\partial w\partial\bar w}\left[-\log(\tilde m)+Cd_M(\cdot)^2\right]
= \frac{\partial^2}{\partial w\partial\bar w}[-\log(\tilde m)]\\
  +\frac{C}{4}\left(2h(z_j,t_j)+(2w_j-2\bar w_j)\frac{\partial}{\partial w_j}h(z_j,w_j)
  -(2w_j-2\bar w_j)\frac{\partial}{\partial\bar w_j}h(z_j,w_j)\right.\\
  \left.+4\Im(w_j)^2\frac{\partial^2}{\partial w_j\partial\bar w_j}h(z_j,w_j)\right)
\end{multline*}
On $M$ we have $\Im(w_j)=0$, hence the transverse laplacian is equal to
\[
\frac{\partial^2}{\partial w\partial\bar w}\left[-\log(\tilde m)+Cd_M(\cdot)^2\right]
=\frac{\partial^2}{\partial w\partial\bar w}[-\log(\tilde m)]+\frac{C}{2}h(z_j,t_j)
\]
that is positive and big if $C$ is big enough. We calculate finally the mixed derivative
\begin{multline*}
\frac{\partial^2}{\partial w\partial\bar z}\left[-\log(\tilde m)+Cd_M(\cdot)^2\right]
=\frac{\partial^2}{\partial w\partial\bar z}[-\log(\tilde m)]\\
+C\left(\frac{-1}{4}(2w_j-2\bar w_j)\frac{\partial}{\partial\bar z}h(z_j,w_j)
+\Im(w_j)^2 \frac{\partial^2}{\partial w\partial\bar z} h(z_j,w_j)\right).
\end{multline*}
On $M$ this mixed derivative is equal to
\[
\frac{\partial^2}{\partial w\partial\bar z}\left[-\log(\tilde m)+Cd_M(\cdot)^2\right]
=\frac{\partial^2}{\partial w\partial\bar z}[-\log(\tilde m)].
\]
We see then that if we take a big constant $C$ the desired inequality is verified on $M$.
\end{proof}

\section{Extension of foliations}\label{sec:extensionglobale}

We show a result on the extension of foliations to strongly pseudoconvex domains in complex algebraic surfaces. We extend the foliation little by little through the levels of the strictly plurisubharmonic exhaustion function. Passing through the non critical levels is classic and consists in the construction of well placed Hartogs figures. We detail the gluing of extensions in the different Hartogs figures on a given level of the exhaustion function. We also detail how to pass through the critical levels of the exhaustion function, which doesn't use the delicate construction of Hartogs figures near critical points, see \cite{Siu-Trautman, LinsNeto, Merker-Porten, Ivashkovich-bochner}.

All this will allow us to extend the CR foliation of a chaotic analytic Levi-flat hypersurface in a complex algebraic surface to a complex analytic global foliation. This has been used by Lins Neto \cite{LinsNeto} to show the non existence of Levi-flat hypersurfaces in complex projective spaces of dimension $\geq 3$. In this case, the connected components of the complement of the Levi-flat hypersurface are Stein by a theorem of Takeuchi \cite{Takeuchi-projectif}.

\begin{thm}\label{thm:extension}
Let $X$ be a complex algebraic surface and $V\Subset X$ a strongly pseudoconvex domain. Let $K\subset V$ be a compact containing the exceptional set $A$ of $V$ and such that $U=V\setminus K$ is connected. Then every holomorphic foliation $\cal G$ on $U$ extends to a holomorphic foliation on $V$.
\end{thm}

By Theorem \ref{thm:complementfortementpconvLP}, the connected components of the complement of a Levi-flat hypersurface whose normal bundle $N_{\cal G}$ has a metric with positive curvature are modifications of Stein domains. By a result of Coltoiu and Mihalache \cite{Coltoiu-Mihalache}, such a connected component has a continuous exhaustion function $\rho:V\to[-\infty,\infty)$, with value $-\infty$ on the exceptional set $A$, that is $\Cinf$ and strictly plurisubharmonic outside $A$. Moreover, perturbing $\rho$ on the set $V\setminus A$ in the topology $\cal C^2$, we can suppose that $\rho|_{V\setminus A}$ is a Morse function, i.e. its critical points are non degenerate and the levels $\rho^{-1}(t)$ have at most one critical point for $t\in\R$. This function will help us to extend the foliation, defined on a neighbourhood of the Levi-flat hypersurface, to the entire complex surface.

We will start by extending the foliation to $V$ without its exceptional set $A$ (sections \ref{sec:passageniveauNC} and \ref{sec:passageniveauC}). Next, we will extend the foliation to $A$ (section \ref{sec:extensionsurA}). Using Remmert's Reduction \cite[Section 2.1]{Peternell}, this will be done by extending a meromorphic function (representing the slope of the leaves) on a singular space.
To do this it will be fundamental to use Hartogs figures and Levi's Extension Theorem \cite{Levi} for meromorphic functions, which we recall now:

\begin{thm}[Levi's Extension Theorem]\label{thm:levi}
Let $X$ be a complex surface. Let $H$ be a Hartogs figure in $X$ and $f$ a meromorphic function on $H$. Then $f$ can be extended to a meromorphic function on $\hat H$.
\end{thm}

We recall that a Hartogs figure $H\subset X$ is an open set biholomorphic to 
\begin{equation}\label{eq:Hartogs}
\{(z_1,z_2)\in\D(R_1)\times\D(R_2) \mid \abs{z_1}>R_1-r_1 \text{ or } \abs{z_2}<r_2)\}\subset\C^2 \, ,
\end{equation}
with $0<r_k<R_k$ for $k=1,2$. Its holomorphic envelope is $\hat H\simeq\D(R_1)\times\D(R_2)$.

For $t\in\R$, we define
\[K_t:=\{\rho\leq t\} \, ,\, V_t:=V\setminus K_t\, ,\]
and
\[E=\{t\in\R \mid \text{there exists a foliation } \cal G_t \text{ on } V_t \text{ such that } \cal G_t|_{U\cap V_t}=\cal G\}.\]
We remark that if $t\in E$, then the extension $\cal G_t$ of $\cal G|_{U\cap V_t}$ to $V_t$ is unique. This is because the exhaustion function $\rho$ doesn't have a local maximum (because it is plusisubharmonic) and hence each connected component of $V_t$ intersects $U$. In particular, for all $t,t'\in E$, with $t<t'$, we have $(\cal G_t)|_{V_{t'}}=\cal G_{t'}$. Hence, if $t_*$ is the infimum of $E$, there exists a singular holomorphic foliation $\cal G_{t_*}$ on $V_{t_*}$ that extends $\cal G_t$ for all $t\in E$. We have to show that $t_*=-\infty$.

\subsection{Passing through non critical levels}\label{sec:passageniveauNC}

In this paragraph we show that $t_*$ can't be a non critical value of $\rho$. This follows from the following classical result.

\begin{lem}\label{lem:passageniveaunoncritique}
If $t$ is a non critical value of $\rho$ and if $\cal G_t$ is a singular holomorphic foliation defined on $V_t$, then there exists $\epsilon>0$ and a foliation $\cal G_{t-\epsilon}$ defined on $V_{t-\epsilon}$ extending $\cal G_t$.
\end{lem}

In order to prove this, we will apply the extension theorems of Hartogs and Levi, that respectively extend holomorphic and meromorphic functions defined on a Hartogs figure to its convex envelope.

We will need the following result, which gives the normal form of strictly plurisubharmonic functions on a neighbourhood of a non critical point. For a proof, see the book of Henkin and Leiterer \cite[Theorem 1.4.14]{Henkin-Leiterer}.

\begin{thm}\label{thm:formenormalenoncritique}
Let $\rho$ be a strictly plurisubharmonic function of class $\cal C^2$ defined on a neighbourhood of 0 in $\C^2$. If $\d\rho(0)\neq0$, then there exists a biholomorphism $h:U\to V$, where $U$ and $V$ are neighbourhoods of 0 in $\C^2$, such that the function $\rho\circ h^{-1}$ is strictly convex (in the real sense) on $V$.
\end{thm}

This theorem allows us to place Hartogs figures near the levels of a strictly plurisubharmonic function:

\begin{prop}\label{prop:boitesdehartogs}
For every $t\in\R$ and every non critical $p\in\rho^{-1}(t)$, there exists an embedding $\iota:\hat{H}\to V$, such that
\begin{enumerate}
\item $\iota(H)\subset V_t,$
\item $\iota(\hat{H})$ contains a neighbourhood of $p$,
\item $\iota(\hat{H})\cap V_t$ is connected.
\end{enumerate}
\end{prop}

\begin{proof}
We choose coordinates $z_1,z_2$ centered at $p$ such that the function $\rho$ is a strictly convex function (in the real sense). Up to a linear change of coordinates, we may assume that the kernel of $d\rho(0,0)$ contains $\C\times\{0\}$.

We consider the Hartogs figure $H$ defined by Equation \eqref{eq:Hartogs} with $R_1,R_2$ very small and $R_2$ much smaller than $r_1$. By the strict convexity of $\rho$, the Hartogs figure $H'=H+(0,\frac{r_2+R_2}{2})$ is completely contained in $V_t$ and its convex envelope $\hat{H'}=\hat{H}+(0,\frac{r_2+R_2}{2})$ contains the origin.

To show that the intersection $\hat{H'}\cap V_t$ is connected, we see that the slices $(\{z_1\}\times\C)\cap(\hat{H'}\cap V_t)$ are, in the coordinate $z_2$, discs $\abs{z_2}<R_2$ without a convex set of $\C$ that do not intersect the disc $\abs{z_2}<\frac{r_2+R_2}{2}$ (because the level $\rho\leq t$ is convex). Such a set is connected and contains the annulus $r_2<|z_2|<\frac{r_2+R_2}{2}$. It follows that the fibered union of these sets (i.e. $\hat{H}+(0,\frac{r_2+R_2}{2})\cap V_t$) is connected.
\end{proof}

\begin{proof}[Proof of Lemma \ref{lem:passageniveaunoncritique}]
We start with the proof of the following result.

\begin{lem}\label{lem:extensionvoisinage}
Let $t\in\R$. For every regular point $p\in\rho^{-1}(t)$ there exists a neighbourhood $W_p$ of $p$ in $V$ such that $\cal G_t$ extends to $V_t\cup W_p$.
\end{lem}

\begin{proof}
Let $(U_p,\phi_p)$ be a chart of $X$ centered at $p$. The function $\tilde\rho:=\rho\circ\phi_p^{-1}$ is strictly plurisubharmonic on a neighbourhood of 0 in $\C^2$. We consider the function $\iota$ of Proposition \ref{prop:boitesdehartogs} that places a Hartogs figure $H$ in $\phi_p(V_t\cap U_p)$. Let $\cal G'$ be the restriction of ${\phi_p}_*(\cal G_t)$ to $H\subset\C^2$. We will prove that there exists a differential 1-form $\omega$ with isolated zeroes on $\hat H$ that defines $\cal G'$ on $H$.

Let $\{U_j\}_{j\in J}$ be a cover of $H$ by flow boxes and $\{\omega_j\}_{j\in J}$ be 1-forms defined on the open sets $U_j$ that give the foliation $\cal G'$ as in Proposition \ref{prop:deffeuilletageformediff}. Taking $(z,w)$ as coordinates in the chart $\phi_p$, we can write
\[\omega_j=g^1_j\d z+g^2_j\d w ,\]
where $g^1_j,g^2_j\in\cal O(U_j)$. By Proposition \ref{prop:deffeuilletageformediff} there exist non vanishing holomorphic functions $h_{jk}$ defined on $U_j\cap U_k$ such that
\[\omega_j=h_{jk}\omega_k.\]
The two last expressions imply that if $U_j\cap U_k\neq\emptyset$ then
\begin{equation}\label{eq:cambiocoordenadas}
g^1_j=h_{jk}g^1_k\quad,\quad g^2_j=h_{jk}g^2_k.
\end{equation}

If $g^l_j$ is identically zero for a $j\in J$, the last equation and the fact that $H$ is connected tell us that $g^l_j$ is zero for all $j\in J$. Since the forms $\omega_j$ are not equally zero, the function $g^l_j$ is not identically zero for at least one $l\in\{1,2\}$ and all $j\in J$. We suppose that $g^1_j$ is not identically zero. Then $\frac{g^2_j}{g^1_j}$ defines a meromorphic function $f_j$ on $U_j$. Equation \eqref{eq:cambiocoordenadas} implies that if $U_j\cap U_k\neq\emptyset$ then $f_j=f_k$ on $U_j\cap U_k$. Hence there exists a meromorphic function $f$ defined on $H$ such that $f|_{U_j}=f_j$. By Levi's Extension Theorem \ref{thm:levi}, this function $f$ extends to a meromorphic function on $\hat H$ that we call $\hat f$. We define a meromorphic 1-form on $\hat H$ by
\[\eta:=\d z+\hat f\d w.\]

Since $\hat H$ is a polydisc, there exists a function $h\in\cal O(\hat H)$ and a holomorphic differential 1-form $\omega$ defined on $\hat H$ with isolated zeroes such that $h\eta=\omega$. We see that for all $j\in J$ there exists $g_j\in\cal O^*(U_j)$ such that $\omega|_{U_j}=g_j\omega_j$. Then the foliation $\tilde{\cal G}$ defined by $\omega$ on $\hat H$ extends $\cal G'$. This extension coincides with ${\phi_p}_*(\cal G_t)$ on the set $\{\tilde\rho>0\}\cap\hat H$ since, by construction of our Hartogs figure, this intersection is connected.

This extension is unique because for any other differential 1-form $\omega'$ on $\hat H$ whose foliation extends ${\phi_p}_*(\cal G_t)$, there exists a non vanishing holomorphic function $a$ on $\phi_p(V_t\cap U_p)$ such that $\omega=a\omega'$. The open set $\hat H$ being connected, the form $\omega\wedge\omega'$ is zero on $\hat H$. Hence there exists a non vanishing holomorphic function $\tilde a$ on $\hat H$ such that $\omega=\tilde a\omega'$. To finish we pullback the foliation $\tilde{\cal G}$ on $V$ by the chart map $\phi_p$. The open set $W_p$ is then $\phi_p^{-1}(\hat H)$.
\end{proof}

By Lemma \ref{lem:extensionvoisinage}, for all $p\in\rho^{-1}(t)$, we have an open set $W_p$ of $V$ containing $p$ such that $\cal G$ extends to $W_p$. Its intersection with $\rho^{-1}(t)$ defines an open set $V'_p$ of $\rho^{-1}(t)$ containing $p$. The family $\{V'_p\}$, for $p$ in $\rho^{-1}(t)$, is then a cover of $\rho^{-1}(t)$. Since this hypersurface is compact, we can choose a finite subcover $\{V'_j\}_{j=1,\ldots,n}$, where $V'_j=V'_{p_j}$.

For $p\in\rho^{-1}(t)$, we define $J_p\subset\{1,\ldots,n\}$ the set of indexes $j$ such that $p\in W_{p_j}$. The intersection $W_{0,p}:=\bigcap_{j\in J_p}W_{j}$ is an open neighbourhood of $p$. We note $r_p:=\mathrm{dist}(p,\partial W_{0,p})$ and we define
\[W_p':=B_p(r_p/3),\qquad W':=\bigcup_{p\in\rho^{-1}(t)}W'_p.\]
We note $W_j:=W_{p_j}$ and $\cal G_j$ the extension of $\cal G$ to $W_j$ given by Lemma \ref{lem:extensionvoisinage}. Let $\omega_j$ be a holomorphic 1-form on $W_j$ that defines the foliation $\cal G_j$.

\begin{lem}\label{lem:extennsionuniquessensintesection}
For all $p\in\rho^{-1}(t)$ and all $j\in J_p$, the foliations $\cal G_j|_{W'_p}$ coincide. More precisely, there exists a holomorphic function $h$, not vanishing on $W'_p\subset W_j\cap W_k$, such that $\omega_j=h_{jk}\omega_k$ on $W_p'$.
\end{lem}

\begin{proof}
The foliations $\cal G_j|_{W'_p}$ and $\cal G_k|_{W'_p}$ coincide on $W'_p\cap V_t$ by construction. Hence $\omega_j\wedge\omega_k=0$ restricted to $W'_p\cap V_t$. Since $W'_p$ is connected, the product of these two forms is still zero on $W'_p$. So there exists a holomorphic function $h$, not vanishing on $W'_p\subset W_j\cap W_k$, such that $\omega_j=h_{jk}\omega_k$ on $W_p'$.
\end{proof}

\begin{lem}\label{lem:extensionVepsilon}
The foliation $\cal G_t$ extends in a unique way to $W'$.
\end{lem}

\begin{proof}
First we show the uniqueness. Let $\cal G_1$ and $\cal G_2$ be two extensions of $\cal G_t$ to $W'$. Since by definition the sets $W'_p$ cover $W'$, it is sufficient to show that $\cal G_1|_{W'_p}=\cal G_2|_{W'_p}$ for all $p\in\rho^{-1}(t)$. The arguments used in the proof of Lemma \ref{lem:extennsionuniquessensintesection} give us the result.

Now we show the existence. We cover $W'$ with the open sets
\[W'_j:=\bigcup_{q\in W_j\cap\rho^{-1}(t)}W_q'\subset W_j.\]
Since $\cal G_t$ extends to $W_j$ then $\cal G_t$ extends to $W'_j$. It is sufficient to show then, that for all $j\neq k\in\{1,\ldots,n\}$ we have $\omega_j=h_{jk}\omega_k$ on $W'_j\cap W'_k$ with $h_{jk}$ a non vanishing holomorphic function.

To do this, we verify that
\[W'_j\cap W'_k=\bigcup_{q\in W_j\cap W_k\cap \rho^{-1}(t)}W_q'.\]
The inclusion $\supset$ come from the definition of $W_j'$. In the other sense, if $x\in W_j'\cap W_k'$, then there exists $q_j,q_k\in\rho^{-1}(t)$ such that $q_j\in W_j$, $q_k\in W_k$ and $x\in W_{q_j}'\cap W_{q_k}'$. To obtain the conclusion, it suffices that $q_j$ or $q_k$ be in $W_j\cap W_k$. If this is not the case, then $d(q_j,q_k)\geq d(q_j,\partial W_j)\geq r_{q_j}$ and also $d(q_j,q_k)\geq d(q_k,\partial W_k)\geq r_{q_k}$. Since by definition $W_{p}' = B_p(r_p/3)$, we would obtain that $W_{q_j}'$ and $W_{q_k}'$ are disjoints, which is a contradiction.

Let $q_0\in W_j'\cap W_k'$. Then there exists $q_1\in W_j\cap W_k\cap \rho^{-1}(t)$ such that $q_0\in W_{q_1}'\subset W_j\cap W_k$. By Lemma \ref{lem:extennsionuniquessensintesection} there exists a holomorphic function $h^1_{jk}$, not vanishing on $W_{q_1}'$, such that $\omega_j=h^1_{jk}\omega_k$ on $W_{q_1}'$. Now, if $q_2\in W_j\cap W_k\cap\rho^{-1}(t)$ and $q_0\in W_{q_2}'$, we also have $\omega_j=h^2_{jk}\omega_k$ on $W_{q_2}'$. We can deduce then
\[h^1_{jk}(q_0)=h^2_{jk}(q_0).\]
For all $q_0\in W_j'\cap W_k'$ we can then define $h_{jk}(q_0):=h^1_{jk}(q_0)$, and this does not depend on the choice of $q_1$. The function $h_{jk}$ is a non vanishing holomorphic function because it satisfies these properties everywhere locally. Since $\omega_j=h_{jk}\omega_k$ on $W_j'\cap W_k'$, this ends the proof.
\end{proof}

\begin{lem}\label{lem:voisinage}
$W'$ is an open neighbourhood of $\rho^{-1}(t)$ in $X$. In particular, there exists $\epsilon>0$ such that $\rho^{-1}(]t-\epsilon,t+\epsilon[)\subset W'$.
\end{lem}

\begin{proof}
For all $q\in\rho^{-1}(t)$ we define
\[d_{\partial W'}(q):=\inf_{q'\in\partial W'}\abs{\rho(q)-\rho(q')}.\]
It is the distance between $q$ and the border of $W'$ measured with the function $\rho$. This function is continuous and defined on the compact $\rho^{-1}(t)$, so it reaches its minimum at a point $q_0\in\rho^{-1}(t)$. This minimum is not zero because the set of critical points of $\rho$ is discrete in $V$. Finally, $W'$ contains $\rho^{-1}(]t-\epsilon,t+\epsilon[)$, where $\epsilon=d_{\partial W'}(q_0)$.
\end{proof}

Lemmas \ref{lem:extensionVepsilon} and \ref{lem:voisinage} conclude the proof of Lemma \ref{lem:passageniveaunoncritique}. 
\end{proof}

\subsection{Passing through critical levels}\label{sec:passageniveauC}

In this paragraph we show that $t_*$ cannot be a critical value of the restriction of $\rho$ to $V\setminus A$. This follows essentially from the fact that the indexes of the function $\rho$ in its critical points can take only the values $0, 1$ or $2$. Remember that the index of a critical point is the maximal dimension of a subspace of the tangent space at this point such that the Hessian of the function is negative. For a strictly plurisubharmonic function on a complex surface such a space cannot have dimension $\geq 3$, otherwise it would contain a complex line on which the Levi form of the function would be negative.

Formally, the passage of critical values will be possible thanks to the following perturbation result.

\begin{lem}\label{lem:perturbation}
Let $p_*$ be a critical point of the exhaustion function $\rho|_{V\setminus A}$ and $I_*$ be a neighbourhood of $t_*=\rho(p_*)$ such that $\rho^{-1}(I_*)$ does not contain any other critical points. Then there exists a continuous function $\rho':V\to[-\infty,+\infty)$ such that
\begin{enumerate}
\item $\rho'$ coincides with $\rho$ on $V\setminus\rho^{-1}(I_*)$.
\item $\rho'$ is smooth and strictly plurisubharmonic on $V\setminus A$.
\item $\rho'$ admits a unique critical point $p'_*$ in $\rho^{-1}(I_*)$ and $\rho'(p_*)>\rho'(p'_*)$.
\item The critical levels $\rho^{-1}(]t_*,+\infty))$ and $\rho'^{-1}(]t'_*, + \infty))$ have connected intersection, where $t'_*=\rho(p'_*)$.
\end{enumerate}
\end{lem}

Before proving this lemma, let us explain why it gives us a contradiction if $t_*$ is supposed to be $>-\infty$ and a critical value of $\rho$. In fact, applying the techniques of Section \ref{sec:passageniveauNC} to the function $\rho'$, we would be able to extend the foliation $\cal G$ to a singular holomorphic foliation $(\cal G)'_{t'_*}$ defined on the critical level $\rho'>t'_*$ (indeed, if we note $t_*^+$ the border of $I_*$ bigger than $t_*$, the foliation $\cal G_{t_*}$ is defined on $(\rho')^{-1}(t_*^+)$ and therefore we could extend it to the next critical level of $\rho'$, i.e. the level $(\rho')^{-1}(t_*')$).

Since the intersection of the levels $\rho^{-1}(]t_*,+\infty))$ and $\rho'^{-1}(]t'_*,+\infty))$ is connected, the foliations $\cal G_{t_*}$ and $(\cal G)'_{t'_*}$ coincide on the intersection of their definition domains. We can then extend the foliation $\cal G$ to a foliation defined on $\rho^{-1}(]t_*,+\infty))\cup\rho'^{-1}(]t'_*,+\infty))$ that contains a neighbourhood of $p_*$. Applying the techniques of Section \ref{sec:passageniveauNC}, this allows us to show that $\cal G$ extends to a foliation on a level of type $\rho^{-1}( ]t_*-\epsilon,+\infty))$ where $\epsilon>0$. This contradicts the minimality of $t_*$.

\begin{proof}[Proof of Lemma \ref{lem:perturbation}]
Since the index of $\rho$ in $p_*$ is $\leq 2$, the Morse Lemma gives us coordinates $(x_1,x_2,y_1,y_2)$ centered at $p_0$ such that \footnote{These real coordinates have no reason to be holomorphic!}
\begin{equation}\label{eq:Morse}
\rho=x_1^2+x_2^2\pm y_1^2\pm y_2^2.
\end{equation}
We consider a neighbourhood $W_*$ of $p_*$ and a number $\epsilon_*>0$ such that the coordinates $(x_1,x_2,y_1,y_2)$ take $W_*$ to the bidisc
\[\D_{\epsilon_*}\times\D_\epsilon:=\{x_1^2+x_2^2\leq\epsilon_*^2\text{ and }y_1^2+y_2^2\leq\epsilon^2\}.\]
We will choose $\epsilon_*>\epsilon>0$ small enough for $W_*$ to be contained in $\rho^{-1}(I_*)$, and we will suppose in what follows that $\epsilon$ is very small with respect to $\epsilon_*^2$.

To simplify notations, we can suppose that $t_*=0$ (we can consider the function $\rho-t_*$).
We introduce
\begin{itemize}
\item a smooth function $\delta:\D_{\epsilon_*}\to\R_+$ such that $\delta(0,0)>0$, $\delta(y_1,y_2)=0$ if $y_1^2+y_2^2\geq\epsilon^2/2$, and with $\cal C^2$-norm smaller than $\epsilon$ 
\item a smooth function $\phi:[0,\epsilon^2_*]\to[0,1]$ that satisfies $\phi=1$ on $[0,\epsilon_*^2/3]$ and $\phi=0$ on $[2\epsilon_*^2/3,\epsilon_*^2]$.
\end{itemize}
We define then
\[\rho''=(x_1-\delta(y_1,y_2))^2+x_2^2\pm y_1^2\pm y_2^2,\]
and 
\[\rho'=(1-\phi(x_1^2+x_2^2))\rho+\phi(x_1^2+x_2^2)\rho''.\]

The function $\rho'$ coincides with $\rho$ outside $W_*$, therefore it satisfies condition one of the lemma. Moreover, it is $\epsilon$ close to $\rho$ in the $\cal C^2$ norm, so that it satisfies condition 2 if $\epsilon$ is small enough. In this case the point $p'_*$ of coordinates $(\delta(0,0),0,0,0)$ is the only critical point of $\rho'$ in $W_*$, therefore in $\rho^{-1}(I_*)$. We have then $\rho'(p_*)=\delta(0,0)^2>0$, hence condition 3 is equally satisfied.

We have to prove condition 4. For this we will take slices $(y_1,y_2)=(c_1,c_2)$ as in the proof of Proposition \ref{prop:boitesdehartogs}, where $c_1^2+c_2^2\leq\epsilon^2$. If $x_1^2+x_2^2\geq\epsilon_*^2/3$ then we have
\[\rho=x_1^2+x_2^2\pm y_1^2\pm y_2^2\geq\epsilon_*^2/3-\epsilon^2>0\] 
and
\[\rho''=(x_1-\delta(y_1,y_2))^2+x_2^2\pm y_1^2\pm y_2^2=\epsilon_*^2/3+O(\epsilon)>0\]
therefore $\rho'>0$. If $x_1^2+x_2^2<\epsilon_*^2/3$, then $\phi=1$, and hence $\rho>0$ and $\rho'>0$ are equations of the exterior of small discs contained in $x_1^2+x_2^2<\epsilon_*^2/3$. In summary, in the slice $(y_1,y_2)=(c_1,c_2)$ the place where $\rho$ and $\rho'$ are $>0$ is described by the exterior of the union of two small discs contained in the disc centered in the origin and of radius $\epsilon_*/\sqrt{3}$. In particular, this shows that the set 
\[\{\rho>0\}\cap\{\rho'>0\}\cap W_*\]
is connected. Now, the set $\{\rho>0\}$ retracts by deformation to $\{\rho>0\}\cap V\setminus W_*$ (by a radial retraction in coordinates $x$). The equality
\[\{\rho>0\}\cap\{\rho'>0\}=\left(\{\rho>0\}\setminus W_*\right)\cup\left(\{\rho>0\}\cap\{\rho'>0\}\cap W_*\right)\]
shows that this set is connected and concludes the proof of Lemma \ref{lem:perturbation}.
\end{proof}

\subsection{Extension through the exceptional set}\label{sec:extensionsurA}

We will need the following two results:

\begin{lem}\cite[Corollary 1.5]{Ivashkovich}\label{lem:extensionfonctionspacenormal}
Let $Y$ be a reduced and normal complex surface. Let $B\subset Y$ be a finite set of points. Then every meromorphic function $f$ on $Y\setminus B$ extends to a meromorphic function on $Y$.
\end{lem}

\begin{prop}[Remmert's Reduction]\cite[Section 2.1]{Peternell}\label{prop:reduccionderemmert}
Let $V$ be a holomorphically convex complex surface. Then there exists a normal Stein space $Y$ and a proper and surjective holomorphic function $\pi:V\to Y$ such that:
\begin{enumerate}
\item The fibers of $\pi$ are connected.
\item $\pi_\ast(\cal O_V)=\cal O_Y$.
\item The canonical map $\cal O_Y(Y)\to\cal O_V(V)$ is an isomorphism.
\item For every holomorphic map $\sigma:V\to Z$ with $Z$ a Stein space there exists a unique holomorphic map $\tau:Y\to Z$ such that $\sigma=\tau\circ\pi$.
\end{enumerate}
\end{prop}

Now we can proceed to the last step of our proof.

\begin{prop}\label{prop:extensionsurA}
Let $X$ be a compact complex surface, let $V\Subset X$ be a strongly pseudoconvex domain and let $A$ be the exceptional set of $V$. Let $\cal G$ be a holomorphic foliation on $V\setminus A$. Then $\cal G$ extends to a holomorphic foliation $\tilde{\cal G}$ on $V$.
\end{prop}

\begin{proof}
The idea is similar to that used to exhibit differential forms defining the foliation in a neighbourhood of singularities: we will extend the slope of the leaves.

First we will construct this slope function outside $A$. Let $\Phi:X\to\C\Proj^N$ be an embedding of the algebraic surface $X$. Let $[z_0:\ldots:z_N]$ be homogeneous coordinates of $\C\Proj^N$. Up to reducing $N$ and composing $\Phi$ by an automorphism of $\C\Proj^N$, we may assume that $\Phi(X)$ is not contained in the hyperplane $\{z_N=0\}$ and hence $\Omega_1:=\Phi^*\d(z_0/z_N)$ and $\Omega_2:=\Phi^*\d (z_1/z_N)$ are a basis of the space of meromorphic differential 1-forms on $X$.

Let $\{U_j\}_{j\in J}$ be a cover of $V\setminus A$ such that on each $U_j$ the foliation $\cal G$ is defined by $\omega_j:=f_j\Omega_0+g_j\Omega_1$, where $f_j$ and $g_j$ are meromorphic functions. On $U_j \cap U_k\neq\emptyset$ we have
\[\Omega_0+\frac{g_j}{f_j}\Omega_1=\Omega_0+\frac{g_k}{f_k}\Omega_1.\]
The collection $\left\{g_k/f_k\right\}$ defines then a meromorphic function on $V\setminus A$, noted $F$. The foliation $\cal G$ is then defined on $V\setminus A$ by
\[\Omega:=\Omega_0+F\Omega_1,\]
and the function $F$ corresponds to the slope of the leaves. Let $\pi:V\to Y$ be the Remmert's Reduction of $V$ and let $\tilde F$ be the meromorphic function defined on $Y\setminus\pi(A)$ by $\tilde F:=F\circ\pi^{-1}$. Since $\pi(A)$ is a finite set of points, and since $V$ is a normal analytic space, $\tilde F$ extends to a meromorphic function on $Y$, noted $\tilde F_{ext}$ (see Lemma \ref{lem:extensionfonctionspacenormal}). We pullback this function to a meromorphic function on $V$ by setting $\hat F:=\tilde F_{ext}\circ\pi$. We define like that a meromorphic 1-form on $V$
\[\hat\Omega:=\Omega_0+\hat F\Omega_1\]
that extends the meromorphic form $\Omega$. This extends the foliation $\cal G$ to the exceptional set $A$.
\end{proof}

\section{Transversely affine foliations}\label{sec:feuilletagestransaffines}

In this part we apply the previous results in the context of transversely affine foliations. In the first section we define these foliations on compact 3-manifolds and we study the case of hyperbolic torus bundles as in \cite{Ghys-Sergiescu}. In the second section we define degenerate transversely affine foliations on complex surfaces as in Scárdua's work \cite{Scardua}. We have chosen to call these foliations \emph{degenerate} instead of \emph{singular} to avoid creating any confusion with a singular holomorphic foliation. The third section is devoted to an extension theorem for transversely affine foliations using the same ideas as in the last section. We show next that a transversely affine Levi-flat hypersurface in a compact algebraic complex surface necessarily has a transverse invariant measure. This allows us to establish that if a hyperbolic torus bundle appears as a Levi-flat hypersurface in a compact algebraic complex surface, then its Cauchy-Riemann foliation necessarily has a compact leaf. Indeed, Ghys and Sergiescu \cite{Ghys-Sergiescu} give a classification of foliations on these bundles: up to conjugation, these foliations either have a compact leaf or they correspond to stable and unstable foliations of Anosov's flow. These last foliations are transversely affine and do not have a transverse invariant measure.

\subsection{Transversely affine foliations on 3-manifolds}

Let $M$ be a compact 3-manifold and $\cal F$ a foliation by Riemann surfaces on $M$. We say that $\cal F$ is \emph{transversely affine} if it has an atlas for which the transversal changes of coordinates $g_{jk}$ of Definition \ref{def:feuilletage3reel} are affine or, equivalently, if there exists a cover $\{U_j\}_{j\in J}$ of $M$ and a family of local submersions $g_j:U_j\to\R$ such that for all $j,k\in J$, on $U_j\cap U_k$ there exist $a_{jk}\in\R^*,b_{jk}\in\R$ such that $g_j=a_{jk}g_k+b_k$. For more details one can consult the book of Godbillon \cite[Chapter III]{Godbillon}.

Before passing to complex surfaces, we study the example of hyperbolic torus bundles of Ghys and Sergiescu \cite{Ghys-Sergiescu}.

\paragraph{Hyperbolic torus bundles and model foliations}\label{par:fibreshyperboliquesentores}

Let $\T^2=\R^2/\Z^2$ be a real torus and let $A$ be a matrix in $\mathrm{GL}(2,\Z)$. Let $\T^3_A$ be the 3-manifold obtained as the quotient of $\T^2\times\R$ by the equivalence relation $(p,0)\sim(Ap,1)$ for all $p\in\T^2$. We call such a manifold a torus bundle over the circle.

We say that this bundle is hyperbolic if $A$ is hyperbolic, i.e. if $\abs{\tr A}>2$. We will restrict our study to the case $\det A=1$ and $\tr A>2$.

A hyperbolic matrix $A$ defines a hyperbolic automorphism of the torus $\T^2$. This automorphism has two eigenvectors of irrational slope. The foliation of the plane $\R^2$ by parallel lines to one of these directions passes to the quotient torus $\T^2$ as a linear foliation. Moreover the foliation produced on $\T^2\times\R$ is invariant by the map $(p,0)\sim(Ap,1)$ and hence it defines a foliation by planes on the bundle $\T^3_A$. We call these foliations the model foliations of $\T^3_A$. They are all transversely affine and moreover

\begin{thm}\cite{Ghys-Sergiescu}\label{thm:Ghysfeuilletages}
Let $\T^3_A$ be an orientable hyperbolic torus bundle. Then every transversely orientable foliation of class $\cal C^r$ on $\T^3_A$ with $r\geq2$ and without compact leaves is $\cal C^{r-2}$ conjugated to one of the model foliations.
\end{thm}

\subsection{Transversely affine foliations on complex surfaces}

In this section we define the notion of degenerate transversely affine foliation on a complex surface as in the papers of Scárdua \cite{Scardua}, Camacho and Scárdua \cite{Camacho-Scardua} and Cousin-Pereira \cite{Cousin-Pereira}. The latter gives a beautiful classification of these foliations that precises Singer's characterization of foliations having a liouvillian first integral.

\subsubsection{Transversely affine foliations}

\begin{deff}
Let $X$ be a complex surface. Let $\cal G$ be a (non-singular) holomorphic foliation on $X$. The foliation $\cal G$ is transversely affine if there exists a foliated atlas $\{(U_j,\varphi_j)\}_{j\in J}$ such that on $U_j\cap U_k\neq\emptyset$ we have
\[\varphi_{jk}(z_k,w_k)=(f_{jk}(z_k,w_k),a_{jk}w_k+b_{jk})=(z_j,w_j)\]
with $a_{jk}, b_{jk}\in\C$, $a_{jk}\neq0$. A singular holomorphic foliation $\cal G$ defined on a complex surface $X$ is transversely affine if it is transversely affine on $X\setminus\sing(\cal G)$.
\end{deff}

\begin{prop}
Let $X$ be a complex surface, $M$ be an analytic Levi-flat hypersurface in $X$ and $\cal F$ the Cauchy-Riemann foliation of $M$. Let $U$ be a neighbourhood of $M$ in $X$ such that $\cal F$ extends to a non-singular holomorphic foliation $\cal G$ on $U$ (such a neighbourhood exists by Proposition \ref{prop:extentionvoisinage}). If $\cal F$ is transversely affine on $M$ then $\cal G$ is transversely affine on $U$.
\end{prop}

\begin{proof}
Let $\{(U_j,(z_j,w_j))\}$ be a family of charts of $M$ such that $M\cap U_j=\{\Im(w_j)=0\}$, see Lemma \ref{lem:formelocaleLeviPlat}. Since the extended foliation $\cal G$ is holomorphic, there exists a sequence $\{c_l\}$ of complex numbers such that
\[w_j=\sum_{l=0}^\infty c_lw_k^l.\]
If $\cal F$ is transversely affine, then for $t_j=\Re(w_j)$ and $t_k=\Re(w_k)$ we have
\[t_j=a_{jk}t_k+b_{jk}\]
where $a_{jk}\in\R^*$ and $b_{jk}\in\R$.
These two equations tell us that $c_0=b_{jk}$, $c_1=a_{jk}$ and $c_l=0$ for all $l\geq 2$. Hence $w_j=a_{jk}w_k+b_{jk}$, the foliation $\cal G$ is transversely affine.
\end{proof}

\subsubsection{Degenerate transversely affine foliations}\label{sec:feuilletagesaffinesdeg}

\begin{deff}\label{prop:formestransaffinecomplexe}
Let $X$ be a complex algebraic surface and $\cal G$ a holomorphic foliation on $X$. Let $\omega$ be a meromorphic 1-form on $X$ defining $\cal G$. We say that $\cal G$ has a degenerate transverse affine structure, or that $\cal G$ is a degenerate transversely affine foliation, if there exists a closed meromorphic 1-form $\eta$ on $X$ such that $\d\omega=\eta\wedge\omega$.
\end{deff}

Logarithmic foliations and Bernoulli foliations on $\Proj\C^2$ are examples of such structures, see \cite{Scardua} for more details. Degenerate transversely affine foliations can also be defined by connections, see the paper of Cousin and Pereira \cite{Cousin-Pereira}. If $\omega$ is a holomorphic section of $N_{\cal G}\otimes\Omega^1_X$, then $\cal G$ is degenerately transversely affine if there exists a divisor $D$ of $X$ and a meromorphic flat connection
\[\nabla:N_{\cal G}\to N_{\cal G}\otimes\Omega^1_X(*D)\]
such that $\nabla(\omega)=0$, where $\Omega^1_X(*D)$ is the sheaf of meromorphic 1-forms on $X$ with poles on $D$. The irreducible components of $D$ are invariant by the foliation $\cal G$. Moreover we have the following property, see \cite[Proposition 2.2]{Cousin-Pereira}

\begin{prop}\label{prop:metcourbnulle}
In $H^2(X,\C)$, the Chern class of the normal bundle $N_{\cal G}$ is equal to
\[c_1(N_{\cal G})=-\sum\alpha_C[C]\]
where the sum is over all the irreducible components of $D$, and where $\alpha_C\in\C$ is the residue $Res_C(\nabla)$ of any meromorphic 1-form defining $\nabla$ in a generic point of $C$.
\end{prop}

\subsection{Extension of transversely affine foliations}

\begin{thm}\label{thm:extensionstructuretransverse}
Let $X$ be an algebraic complex surface and $V\Subset X$ a strongly pseudoconvex domain. Let $K\subset V$ be a compact such that $U=V\setminus K$ is connected and contains the exceptional set of $V$. Then every regular transversely affine foliation $\cal G$ on $U$ extends to a degenerate transversely affine foliation on $V$.
\end{thm}

By Theorem \ref{thm:extension} the foliation $\cal G$ defined on $U$ extends to a foliation $\tilde{\cal G}$ on $V$. Let $\omega$ be a meromorphic 1-form on $X$ defining $\tilde{\cal G}$. 

\begin{lem}\label{prop:sectionstransaffinecomplexe}
There exists a closed meromorphic 1-form $\eta$ on $U$ such that $\d\omega=\eta\wedge\omega$ on $U$.
\end{lem}

\begin{proof}
Let $\{U_j\}_{j\in J}$ be a cover of $U$ by flow boxes of the transversely affine foliation $\cal G$. Let $\{g_j:U_j\to\C\}$ be a collection of holomorphic local submersions defining the transverse affine structure of $\cal G$ on $U$. On each $U_j\cap U_k\neq\emptyset$, we have
\begin{equation}\label{eq:affine}
g_j=a_{jk}g_k+b_{jk} \quad \text{ where } \quad (a_{jk},b_{jk})\in\C^*\times\C.
\end{equation}
Let $f=(f_j)_{j\in J}$ be a meromorphic section of $N_{\cal G}$ associated to $\omega$ such that $\omega|_{U_j}=f_j\d g_j$ on $U_j$ (see Proposition \ref{prop:sectmerom}). We will verify that the meromorphic 1-forms $\frac{\d f_j}{f_j}$ glue together on $U$. Since $\omega|_{U_j}$ come from the global differential form $\omega$, on each $U_j\cap U_k\neq\emptyset$ we have
\begin{equation}\label{eq:local}
f_j\d g_j=f_k\d g_k.
\end{equation}
Derivating Equation \eqref{eq:affine} and replacing the result in \eqref{eq:local} we obtain $f_ja_{jk}=f_k$. Thus
\[\frac{\d f_j}{f_j}=\frac{\d f_k}{f_k}.\]
We can then define a meromorphic 1-form $\eta$ on $U$ by $\eta|_{U_j}=\frac{\d f_j}{f_j}$. Moreover we have $\d\omega|_{U_j}=\d f_j\wedge\d g_j=(\d f_j/f_j)\wedge f_j\d g_j=\eta|_{U_j}\wedge\omega|_{U_j}$, hence $\d\omega=\eta\wedge\omega$ in $U$.
\end{proof}

The form $\eta$ extends to $V$ because $V$ is strongly pseudoconvex. To see this, it is sufficient to apply to $\eta$ the methods from section \ref{sec:extensionglobale}. This establishes Theorem \ref{thm:extensionstructuretransverse}.

\subsection{Existence of invariant measures}\label{sec:constructionmetrique-mesuretrasinv}

\begin{thm}\label{thm:mesuretransinv}
Let $X$ be a complex algebraic surface. Let $M$ be an analytic Levi-flat hypersurface of $X$. We suppose that the Cauchy-Riemann foliation of $M$ is transversely affine. Then this foliation has a transverse invariant measure.
\end{thm}

\begin{proof}
We suppose that $\cal F$, the Cauchy-Riemann foliation on $M$, does not have a transverse invariant measure. By Theorem \ref{thm:metriquecourburepositive}, the normal bundle $N_{\cal F}$ of $M$ has a metric with positive curvature in the direction of the leaves that we note $m_{\cal F}$.

By Section \ref{sec:extensionglobale}, the foliation $\cal F$ defined on the Levi-flat hypersurface $M$ extends to a holomorphic foliation $\tilde{\cal G}$ on the whole surface $X$. By Theorem \ref{thm:extensionstructuretransverse} this extended foliation is transversely affine on a neighbourhood of $M$ and degenerately transversely affine on $X$. By Proposition \ref{prop:metcourbnulle} the bundle $N_{\tilde{\cal G}}\otimes\cal O(\sum\alpha_C[C])$ is flat on $X$, where $C$ denotes the irreducible components of the divisor $D$ where the transverse structure degenerates. Since $X$ is kählerian, this bundle has a metric with trivial curvature by the $\partial\bar\partial$-Lemma. Since the irreducible components $C$ of $D$ do not meet a neighbourhood of $M$, the bundle $N_{\cal F}$ has, on $M$, a metric $m_0$ with trivial curvature.

The metrics $m_{\cal F}$ and $m_0$ on $N_{\cal F}$ are related by $m_{\cal F}=e^{-\tau}m_0$, where $\tau:M\to\R$ is a function of class $\Cinf$. The curvatures of these two metrics are related then on $M$ by the equation
\begin{equation}\label{eq:RelationCourbures-dem2}
\frac{i}{2\pi}\partial\bar\partial_{\cal F}(-\log m_{\cal F})=\frac{i}{2\pi}\partial\bar\partial_{\cal F}\tau+\frac{i}{2\pi}\partial\bar\partial_{\cal F}(-\log m_0).
\end{equation}
The last term on the right hand side is zero because it is equal to the curvature of the metric $m_0$. Now, since $M$ is compact and $\tau$ is continuous, the function $\tau$ reaches its maximum on $M$ at a point $p$. The restriction of $\tau$ to the leaf of $\cal F$ passing by $p$ reaches also its maximum at $p$ and then $i\partial\bar\partial_{\cal F}\tau$ is negative on $p$. We obtain a contradiction because the left term in the equation \eqref{eq:RelationCourbures-dem2} is strictly positive.
\end{proof}

As we announced in the introduction, combining this result with the following theorem of Ghys, which gives the nature of transverse invariant measures on foliated 3-manifolds that are transversely affine, we obtain that a transversely affine Levi-flat hypersurface in an algebraic complex surface is quasi-periodic or has an algebraic curve.

\begin{thm}\label{thm:classificationmesuresGhys}\cite{Ghys}
Let $M$ be a compact analytic 3-manifold and let $\cal F$ be a codimension 1 analytic foliation on $M$. We suppose that $\cal F$ is transversely affine and that it has an ergodic transverse invariant measure $\mu$. Then $\mu$ is of one of the following types:
\begin{enumerate}
\item $\mu$ is supported on a compact leaf of $\cal F$.
\item $\cal F$ is Riemannian, i.e. there exists a riemannian metric on the normal bundle of $\cal F$ that is invariant by holonomy and $\mu$ is the volume measure associated to this metric.
\end{enumerate}
\end{thm}

\subsection{Hyperbolic torus bundles as Levi-flat hypersurfaces}

\begin{thm}
Let $X$ be an algebraic complex surface. Let $M$ be an analytic Levi-flat hypersurface on $X$ diffeomorphic to a hyperbolic torus bundle. Then the Cauchy-Riemann foliation on $M$ has at least one compact leaf.
\end{thm}

\begin{proof}
If the Cauchy-Riemann foliation on $M$ does not have a compact leaf, then it is conjugated to the stable or unstable foliation on $M$ by Ghys and Sergiescu's Theorem \ref{thm:Ghysfeuilletages}. These foliations are transversely affine and do not have a transverse invariant measure. Theorem \ref{thm:mesuretransinv} gives us then a contradiction.
\end{proof}

Under the hypotheses of the last theorem, we can ask if the Cauchy-Riemann foliation has several compact leaves, or even if it is a fibration in compact Riemann surfaces. The following proposition, communicated to us by Étienne Ghys, shows that this is not always the case (in particular, \cite[Proposition 2, p194]{Ghys-Sergiescu} is not true).

\begin{prop}
Every hyperbolic torus bundle $\T^3_A$ admits foliations by Riemann surfaces presenting simultaneously compact and non compact leaves.
\end{prop}

\begin{proof}
Let us view $\T^3_A$ as a quotient of a Lie group $G$ of dimension 3 by a lattice $\Gamma\subset G$. Let $G=\R\ltimes\R^2$ where the semi-direct product is given by the action of $\R$ on $\R^2$
\[t\cdot(x,y)=(e^tx,e^{-t}y),\]
Let $A$ be a hyperbolic matrix. Let $\Lambda$ be a lattice invariant by the matrix $A$. We define $\Gamma=\Z\log\lambda\ltimes\Lambda$. Then
\[\R^2/\Lambda\to\Gamma\backslash G\to\Z\log\lambda\backslash\R\]
is a torus bundle. We define 3 one-parameter groups
\begin{align*}
\phi_t:\R\to G:t\mapsto(t,0,0)\\
\phi_x:\R\to G:x\mapsto(0,x,0)\\
\phi_y:\R\to G:y\mapsto(0,0,y)
\end{align*}
that give us respective vector fields $Z,X$ and $Y$, which applied to a $\Cinf$-function $h:\Gamma\backslash G\to\R$ and calculated at a point $p=(t_0,x_0,y_0)$ give
\[Z\cdot h(p)=\frac{\partial h}{\partial t}(p),\qquad
X\cdot h(p)=e^{t_0}\frac{\partial h}{\partial x}(p),\qquad
Y\cdot h(p)=e^{-t_0}\frac{\partial h}{\partial y}(p).\]

These satisfy moreover the commutator relations:
\[[Z,X]=X,\,[Z,Y]=-Y \text{ and } [X,Y]=0.\]

These vector fields give us a frame of the tangent bundle of the group $G$. With them we can define foliations by giving integrable distributions of planes, for example
\begin{equation*}
\R Z+\R X,\,\R Z+\R Y,\,\R X+\R Y,
\end{equation*}
that correspond respectively to unstable, stable and trivial foliations. With these three fields we can build a foliation without Reeb component and with compact leaves on our torus bundle. Indeed, let $f:G\to\R$ be an analytic function that depends only on $t$ and let
\[\cal C:=\R X+\R(Z+f(t)Y)\]
be a distribution of planes on $G/\Gamma$. We verify that this is an integrable distribution:
\begin{align*}
[X,Y+f(t)Z]
 &= [X,Y]+[X,f(t)Z]\\
 &= X(f(t)Z)-f(t)ZX\\
 &= (X\cdot f)Z+f(t)[X,Z]\\
 &= (X\cdot f)Z-f(t)X\quad(\text{since }[X,Z]=-X)\\
 &= -f(t)X\quad\text{(since $f$ does not depend on $t$).}
\end{align*}
So $[X,Y+f(t)Z]=-f(t)X$ belongs to $\cal C$ and our distribution is integrable. This foliation has compact leaves for all $t\in\R$ such that $f(t)=0$. They are tori of the fibration. It does not have Reeb components: between two compact leaves the leaves are dense and wind around the compact leaves.
\end{proof}

\bibliographystyle{alpha-abbrvsort}
\bibliography{mybib}

\end{document}